\documentclass{amsart}

\newtheorem{theorem}{Theorem}[section]
\newtheorem{lemma}[theorem]{Lemma}

\theoremstyle{definition}

\newtheorem{prop}[theorem]{Proposition}

\newtheorem{cor}[theorem]{Corollary}

\usepackage{amssymb,latexsym}
\usepackage{enumerate}

\theoremstyle{remark}

\numberwithin{equation}{section}

\begin{document}

\baselineskip=16pt

\title[Improved subconvexity bounds]
{Improved subconvexity bounds for
$GL(2)\times GL(3)$ and $GL(3)$ $L$-functions
by weighted stationary phase}

\author{Mark McKee \and Haiwei Sun
\and Yangbo Ye}

\address{Mark McKee${}^1$: mark-mckee@uiowa.edu}

\address{Haiwei Sun${}^2$: hwsun@sdu.edu.cn}

\address{Yangbo Ye${}^{1,3}$: yangbo-ye@uiowa.edu}

\address{${}^1$
Department of Mathematics, The University of Iowa,
Iowa City, Iowa 52242-1419, United States}

\address{${}^2$
School of Mathematics and Statistics, Shandong University,
Weihai, Shandong 264209, China}

\address{${}^3$
Corresponding author}

\address{These authors contributed equally to this work.}

\subjclass[2010]{11F66, 11M41, 41A60}



\keywords{
$GL(3)$;
$GL(3)\times GL(2)$;
automorphic $L$-function;
Rankin-Selberg $L$-function;
subconvexity bound;
first derivative test;
weighted stationary phase}
\thanks{The second author is partially supported by the
National Natural Science Foundation of China (Grant No. 11601271), and China Postdoctoral Science Foundation Funded Project  (Project No. 2016M602125).}

\begin{abstract}
Let $f$ be a fixed self-contragradient Hecke-Maass form for
$SL(3,\mathbb Z)$, and $u$ an even Hecke-Maass form for
$SL(2,\mathbb Z)$ with Laplace eigenvalue $1/4+k^2$, $k\geq0$.
A subconvexity bound $O\big((1+k)^{4/3+\varepsilon}\big)$
in the eigenvalue aspect is
proved for the central value at $s=1/2$ of the Rankin-Selberg
$L$-function $L(s,f\times u)$. Meanwhile, a subconvexity
bound $O\big((1+|t|)^{2/3+\varepsilon}\big)$ in the $t$
aspect is proved for $L(1/2+it,f)$. These bounds improved
corresponding subconvexity bounds proved by Xiaoqing Li
(Annals of Mathematics, 2011). The main techniques in the
proofs, other than those used by Li, are $n$th-order
asymptotic expansions of exponential integrals in the cases
of the explicit first derivative test, the weighted first
derivative test, and the weighted stationary phase integral,
for arbitrary $n\geq1$. These asymptotic expansions
sharpened the classical results for $n=1$ by Huxley.
\end{abstract}

\maketitle

\section{Introduction}\label{sec:intro}

Bounds for automorphic $L$-functions on the critical line
Re$(s)=1/2$ are central questions in number theory and have
far-reaching applications (cf. Iwaniec and Sarnak \cite{IwncSrnk}
and Michel \cite{Mchl}). The ultimate conjectured bounds are
predicted by the Lindel\"of Hypothesis, while trivial bounds include
the convexity bounds as a consequence of the
Phragm\'en-Lindel\"of principle. Any bound which have a power
saving over the corresponding convexity bound is highly
non-trivial and called a subconvexity bound.

The strength of a subconvexity bound is crucial. There are
important applications which depend on the strength of the
subconvexity bounds. A notable example is the number of
real zeros of a holomorphic Hecke cusp form $f$ for
$SL(2,\mathbb Z)$ of weight $k$, i.e., zeros
of $f$ on $\{iy|y\geq1\}$. By Ghosh and Sarnak \cite{GhshSrnk},
the number of such zeros is $\gg\log k$. Their proof uses
a Weyl-like, i.e., a $1/3$ power-saving, subconvexity bound
for $L(s,f)$ proved by Peng \cite{Peng} and Jutila and Motohashi
\cite{JutlMoto}. Note that a subconvexity bound for $L(s,f)$
with a power saving less than $1/3$ does not suffice in
\cite{GhshSrnk}.

In this paper, we will prove subconvexity bounds for certain
Rankin-Selberg $L$-functions for $GL(3)\times GL(2)$ and
automorphic $L$-functions for $GL(3)$ over $\mathbb Q$ which
improve bounds established by Xiaoqing Li \cite{XLi1}.

\begin{theorem}\label{SumIntBdd}
Let $f$ be a fixed self-contragradient Hecke-Maass form for
$SL(3,\mathbb Z)$ normalized by $A(1,1)=1$, and $\{u_j\}$ an
orthonormal basis of even
Hecke-Maass forms for $SL(2,\mathbb Z)$. Denote by $1/4+t_j^2$,
$t_j\geq0$, the Laplace eigenvalue of $u_j$. Then for large
$T$ and $T^{1/3+\varepsilon}\leq M\leq T^{1/2}$ we have
\begin{eqnarray}\label{SumInt<<}
&&
\sum_j
e^{-(t_j-T)^2/M^2}
L\Big(\frac12,f\times u_j\Big)
+
\frac1{4\pi}
\int_{-\infty}^\infty
e^{-(t-T)^2/M^2}
\Big|L\Big(\frac12-it,f\Big)\Big|^2dt
\ll_{\varepsilon,f}
T^{1+\varepsilon}M
\end{eqnarray}
for any $\varepsilon>0$.
\end{theorem}

Note that in \cite{XLi1} the same \eqref{SumInt<<} was proved for
$T^{3/8+\varepsilon}\leq M\leq T^{1/2}$. As pointed out in
\cite{XLi1},
\begin{equation}\label{>=0}
L\Big(\frac12,f\times u_j\Big)
\geq0
\end{equation}
was proved by Lapid \cite{Lpd} because $f$ is orthogonal and
$u_j$ is symplectic (Jacquet and Shalika \cite{JcqtShlk}). 
The nonnegativity in \eqref{>=0} allows us to deduce a bound for 
individual terms from \eqref{SumInt<<}. 

We remark that the normalization of $u_j$ is different from
the normalization $\lambda_{u_j}(1)=1$ as required in the
definition of $L(s,f\times u_j)$, but the discrepancy is within
$t_j^\varepsilon$ as proved in Hoffstein and Lockhart
\cite{HffsLckh}.
The smaller allowable power of $T$ for $M$ in Theorem
\ref{SumInt<<} gives us a smaller subconvexity bound.

\begin{cor}\label{SubBdd}
{\it Let $f$ be a fixed self-contragradient Hecke-Maass form for
$SL(3,\mathbb Z)$ normalized by $A(1,1)=1$, and $u$ an even Hecke-Maass form for
$SL(2,\mathbb Z)$ normalized by $\lambda_{u}(1)=1$. Denote
by $1/4+k^2$, $k>0$, the Laplace eigenvalue of $u$. }
$$
L\Big(\frac12,f\times u\Big)
\ll_{\varepsilon,f}
k^{4/3+\varepsilon}.
$$
\end{cor}

Note that Corollary \ref{SubBdd} improved the bound
$O\big(k^{11/8+\varepsilon}\big)$ proved in \cite{XLi1}.
The convexity bound is $O\big(k^{3/2+\varepsilon}\big)$.
Because of the nonnegativity \eqref{>=0}, the bound in
\eqref{SumInt<<} implies a square moment bound for
$L(s,f)$ over a short interval.

\begin{cor}\label{moment}
{\it Let $f$ be a fixed self-contragradient Hecke-Maass form for
$SL(3,\mathbb Z)$ normalized by $A(1,1)=1$. Then for
$T^{1/3+\varepsilon}\leq M\leq T^{1/2}$ }
\begin{equation}\label{moment<<}
\int_{-\infty}^\infty
e^{-(t-T)^2/M^2}
\Big|L\Big(\frac12-it,f\Big)\Big|^2dt
\ll_{\varepsilon,f}
T^{1+\varepsilon}M.
\end{equation}
\end{cor}

Since $f$ is a $GL(3)$ form, the square moment in
\eqref{moment<<} is comparable to a sixth power moment of
the Riemann zeta function. Similar arguments were carried
out for a $GL(2)$ form in Ye \cite{YeY} and Lau, Liu and Ye
\cite{LauLiuYeY}.

By a standard argument of analytic number theory (cf.
Heath-Brown \cite{HthB} or Ivi\'c \cite{Ivic}, p. 197),
we derived a
subconvexity bound for $L(s,f)$ in the $t$ aspect. Its
improvement over \cite{XLi1}'s
$O\big((1+|t|)^{11/16+\varepsilon}\big)$
is again based on the smaller allowable power of $T$ for $M$.
The convexity bound is $O\big((1+|t|)^{3/4+\varepsilon}\big)$.

\begin{cor}\label{tSubBbb}
{\it Let $f$ be a fixed self-contragradient Hecke-Maass form for
$SL(3,\mathbb Z)$ normalized by $A(1,1)=1$. Then}
$$
L\Big(\frac12+it,f\Big)
\ll_{\varepsilon,f}
(1+|t|)^{2/3+\varepsilon}.
$$
\end{cor}

Following Ye and Deyu Zhang \cite{YeYdyZh}, we can deduce
the following result on zero density for $L(s,f)$ from
\eqref{moment<<}. Let
$$
N_f(\sigma, T, T+T^\delta)
=
\#\big\{\rho=\beta+i\gamma\big|
L(\rho,f)=0, \sigma<\beta<1, T\leq\gamma\leq T+T^\delta
\big\}
$$
be the number of zeros of $L(s,f)$ in the box of
$\sigma<\beta<1$ and $T\leq\gamma\leq T+T^\delta$.

\begin{cor}\label{0density}
{\it Let $f$ be a fixed self-contragradient Hecke-Maass form for
$SL(3,\mathbb Z)$. Then for $1/3<\delta\leq1$, we have}
\begin{eqnarray}\label{0density<<}
N_f(\sigma, T, T+T^\delta)
&\ll_{\varepsilon,f}
T^{\frac{(2+4\delta)(1-\sigma)}{3-2\sigma}+\varepsilon}
&
for\ 1/2\leq\sigma< \frac{2+\delta}{2+2\delta} ;
\\
&\ll_{\varepsilon,f}
T^{2(1+\delta)(1-\sigma)+\varepsilon}
&
for\ \frac{2+\delta}{2+2\delta}\leq\sigma<1.
\nonumber
\end{eqnarray}
\end{cor}

We note that Corollary \ref{0density} shows that
\eqref{0density<<} is now valid on a shorter interval
$[T,T+T^\delta]$ with
$1/3<\delta\leq1$ than the interval with $3/8<\delta\leq1$ in
\cite{YeYdyZh} which uses Li \cite{XLi1}.

As noted in \cite{XLi1}, Theorem \ref{SumIntBdd} can also be
proved for $f$ being the minimal Eisenstein series on $GL(3)$.
This has been carried out in Lu \cite{QLu}. Our proof and
improvement can also be applied to that case.

P. Sarnak pointed out to us that for a holomorphic cusp 
form $g$ for $SL(2,\mathbb Z)$, the Dirichlet series 
for the $L$-functions $L(s,Sym^2g)$ and $L(s,Sym^2g\times u_j)$ 
have the same structure and properties as $L(s,f)$ and 
$L(s,f\times u_j)$, respectively, for $f$ being a self-dual 
Maass form for $SL(3,\mathbb Z)$ (cf. Bump \cite{Bmp1,Bmp2} 
and Luo and Sarnak \cite{LuoSrnk}). Consequently our 
theorem and corollaries are also valid for such 
$L(s,Sym^2g)$ and $L(s,Sym^2g\times u_j)$.

The main techniques of our proof, other than those used
in \cite{XLi1}, include an asymptotic expansion of 
exponential integrals
\begin{eqnarray} \label{Ige(f)}
&&
\int_{\alpha}^{\beta} g(x)e(f(x))~dx
\end{eqnarray}
when $f'(x)$ changes signs
at a point $x=\gamma$ with $\alpha<\gamma<\beta$.
Huxley \cite{Hxly} obtained the first-order asymptotic
expansion of \eqref{Ige(f)}. His results \cite{Hxly}
are used widely as standard techniques in analytic number
theory and other branches of mathematics.

What we need in our proof, however, is an asymptotic expansion
of \eqref{Ige(f)} beyond the first order.
Blomer, Khan and Young \cite{BlmKhnYng} proved such an
asymptotic expansion for $f(x)$ being smooth and $g(x)$ 
being smooth of compact support. In \cite{McKSunYe} we proved 
a similar asymptotic expansion for $f(x)$ being continuously 
differentiable $2n+3$ times and $g(x)$ being continuously 
differentiable $2n+1$ times on a finite interval 
$[\alpha,\beta]$. Since the latter one is explicitly 
written, we will use it in the present paper:
\begin{eqnarray*}
\int_{\alpha}^{\beta}g(x)e( f(x))~dx
&=&
\frac{e(f(\gamma)\pm1/8)}{\sqrt{|f''(\gamma)|}}
\Big(g(\gamma)+\sum_{j=1}^{n}\varpi_{2j}\frac{(-1)^{j}(2j-1)!!}{(2\pi i f''(\gamma))^j}
\Big)
\\
&&
+  \text{Boundary terms} + \text{Error terms}.
\end{eqnarray*}
Here $\gamma$ is the only zero of $f'(x)$ in $(\alpha,\beta)$, 
and $\varpi_{2j}$ are given in \eqref{varpi-5}. Note that 
the boundary terms do not appear in \cite{BlmKhnYng}. 
See Proposition \ref{WSPI} below for detail. We will
apply Voronoi's summation formula (Lemma \ref{propVoronoi}) 
and its asymptotic expansion (Lemma \ref{Vasymptotics}) to 
the leading term of \eqref{varpi-5} 
for all $\varpi_{2j}$ the second time.

In the following sections, $\varepsilon$ is any arbitrarily
small positive number. Its value may be different on each
occurrence.

\section{Oscillatory integrals}

The following proposition is the weighted first derivative
test, which strengthens Lemma 5.5.5 of \cite{Hxly},
p.113, with more boundary terms and smaller error terms. We 
can also use a similar formula proved in Jutila and Motohashi 
\cite{JutlMoto}, Lemma 6.

\begin{prop}\label{WFDT} (McKee, Sun and Ye \cite{McKSunYe}) 
Let $f(x)$ be a real-valued function, $n+2$ times continuously differentiable for $\alpha\leq x\leq\beta$, and $g(x)$ a 
real-valued function, $n+1$ times continuously differentiable for $\alpha\leq x\leq\beta$. Suppose that there are positive parameters $M$,
$N$, $T$, $U$, with $M\geq\beta-\alpha$, and positive constants $C_r$ such that for $\alpha\leq x\leq\beta$,
\begin{equation*}
|f^{(r)}(x)|\leq C_r\frac{T}{M^r},\ |g^{(s)}(x)|\leq C_s\frac{U}{N^s},
\end{equation*}
for $r=2,\ldots,n+2$, and $s=0,\ldots,n+1$. If $f^\prime(x)$ and $f^{\prime\prime}(x)$ do not change signs on the interval $[\alpha,\beta]$,
then we have
\begin{eqnarray*}
\int_{\alpha}^{\beta}g(x)e(f(x))dx&=&\Big[e(f(x))\sum_{i=1}^{n}H_{i}(x)\Big]_{\alpha}^{\beta}
\nonumber
\\
&&+\label{error1}
O\Big(\frac{M}{N}\sum_{j=1}^{[n/2]}\frac{UT^j
}{\min|f^\prime|^{n+j+1}M^{2j}}\sum_{t=j}^{n-j}\frac{1}{N^{n-j-t}M^t}\Big)
\\
&&+\label{error2}
O\Big(\Big(\frac{M}{N}+1\Big)\frac{U}{N^{n}\min{|f^\prime|^{n+1}}}\Big)
\\
&&+\label{error3}
O\Big(\sum_{j=1}^{n}\frac{UT^j}{\min|f^\prime|^{n+j+1}M^{2j}}\sum_{t=0}^{n-j}\frac{1}{N^{n-j-t}M^{t}}\Big),
\end{eqnarray*}
where
\begin{equation}\label{the2-1}
H_1(x)=\frac{g(x)}{2\pi if^\prime(x)},\ H_i(x)=-\frac{H_{i-1}^\prime(x)}{2\pi if^\prime(x)}
\end{equation}
for $i=2,\ldots,n$.
\end{prop}

The following proposition is for a weighted stationary phase 
integral and sharpens Lemma 5.5.6 of \cite{Hxly}, p.114, 
with main terms up to the $n$th order, more boundary
terms and smaller error terms. In \cite{BlmKhnYng}, 
Proposition 8.2 Blomer, Khan and Young obtained the 
same main terms and the last big-$O$ term as in 
\eqref{thm3}, under the assymptions that $f(x)$ and $g(x)$ 
are smooth and $g(x)$ is compactly supported on $\mathbb R$. 
We may use their version in the present paper. 

\begin{prop}\label{WSPI} (McKee, Sun and Ye \cite{McKSunYe})
Let $f(x)$ be a real-valued function, $2n+3$ times continuously
differentiable for $\alpha\leq x\leq\beta$, and $g(x)$ a
real-valued
function, $2n+1$ times continuously differentiable for $\alpha\leq x\leq\beta$.
Let  $H_k(x)$ be defined as in \eqref{the2-1}.
Assume that there are positive parameters $M$,
$N$, $T$, $U$ with
\begin{equation}\label{M-N}
M\geq\beta-\alpha,
\end{equation}
and positive constants $C_r$ such that for $\alpha\leq x\leq\beta$,
\begin{equation}\label{UfLfUg}
|f^{(r)}(x)|\leq C_r\frac{T}{M^r},
\ |f^{(2)}(x)|\geq \frac{T}{C_2M^2},
\ |g^{(s)}(x)|\leq C_s\frac{U}{N^s},
\end{equation}
for $r=2,\ldots,2n+3$, and $s=0,\ldots,2n+1$. 
Suppose that $f'(x)$ changes signs only at $x=\gamma$, from negative to positive, with $\alpha<\gamma<\beta$. 
Let
\begin{equation*}\label{Deltamin}
\Delta=
\min\Big\{\frac{\log2}{C_2},\frac{1}{C_2^2\max\limits_{2\leq k\leq 2n+3}\{C_k\}}\Big\}.
\end{equation*}
If $T$ is sufficiently large such that 
$T^{\frac{1}{2n+3}}\Delta>1$, we have for $n\geq2$ that
\begin{eqnarray*}\label{thm3}
&&
\int_{\alpha}^{\beta}g(x)e(f(x))dx
\\
&=&
\frac{e\Big(f(\gamma)+\frac{1}{8}\Big)}{\sqrt{f''(\gamma)}}\Big(g(\gamma)+\sum_{j=1}^{n}\varpi_{2j}\frac{(-1)^{j}(2j-1)!!}{(4\pi i\lambda_2)^j}
\Big)
+\Big[e(f(x))\cdot\sum_{i=1}^{n+1}H_{i}(x)\Big]_{\alpha}^{\beta}\nonumber
\\
&&+
O\Big(\frac{UM^{2n+5}}{T^{n+2}N^{n+2}}\Big(\frac{1}{(\gamma-\alpha)^{n+2}}+\frac{1}{(\beta-\gamma)^{n+2}}\Big)\Big)\nonumber
+
O\Big(\frac{UM^{2n+4}}{T^{n+2}}
\Big(\frac{1}{(\gamma-\alpha)^{2n+3}}+\frac{1}{(\beta-\gamma)^{2n+3}}\Big)\Big)\nonumber
\\
&&+O\Big(\frac{UM^{2n+4}}{T^{n+2}N^{2n}}\Big(\frac{1}{(\gamma-\alpha)^{3}}+\frac{1}{(\beta-\gamma)^{3}}\Big)\Big)+
O\Big(\frac{U}{T^{n+1}}\Big(\frac{M^{2n+2}}{N^{2n+1}}+M\Big)\Big),
\nonumber
\end{eqnarray*}
where
$$
\lambda_j=\frac{f^{(j)}(\gamma)}{j!}
\ \text{for}\ j=2,\ldots,2n+2,
\hspace{3mm}
\eta_\ell=\frac{g^{(\ell)}(\gamma)}{\ell!}
\ \text{for}\ \ell=0,\ldots,2n,
$$
and
\begin{equation}\label{varpi-5}
\varpi_k
=\eta_k
+
\sum_{\ell=0}^{k-1}
\eta_\ell
\sum_{j=1}^{k-\ell}
\frac{C_{k\ell j}}{\lambda_{2}^j}
\sum_{{\mbox{\tiny$\begin{array}{c}
3\leq n_1,\ldots,n_j\leq 2n+3\\n_1+\cdots+n_j=k-\ell+2j \end{array}$}}}\lambda_{n_1}\cdots\lambda_{n_j},
\end{equation}
with $C_{k\ell j}$ being some constant coefficients.
\end{prop}

\section{Background on automorphic forms}
We will follow the setting and notations in Li \cite{XLi1}.
Recall for $m,n \geq 1$ the Kuznetsov trace formula
(Kuznetsov \cite{Kznt} and Conrey and Iwaniec \cite{CnryIwnc})
\begin{eqnarray}\label{kuznetsov}
&&
{\mathop{\sum\nolimits'}
\limits_{j \geq 1}}
h(t_j) \omega_j \lambda_j(m) \lambda_j(n) +
\frac{1}{4\pi} \int_\mathbb{R} h(t) \omega(t)
\bar{\eta} \Big(m, \frac{1}{2}+it\Big)
\eta\Big(n, \frac{1}{2}+it\Big) ~dt   
\\
&=&
\delta(m,n)\frac{H}{2} + \sum_{c \geq 1} \frac{1}{2c}
\Big\{ S(m,n;c)
H^+\Big(\frac{4\pi\sqrt{mn}}{c}\Big)
+
S(-m,n;c)H^-\Big(\frac{4\pi\sqrt{mn}}{c}\Big)
\Big\}.
\nonumber
\end{eqnarray}
Here $\sum'$ in \eqref{kuznetsov} means we are only summing over 
even Maass forms $u_j$, $\delta(m,n)$ is the Kronecker delta,
\begin{equation}\label{omegaomega}
\omega_j = \frac{4\pi |\rho_j(1)|^2}{\cosh \pi t_j}, ~~~
\omega(t) = 4\pi \frac{|\phi(1,{1}/{2}+it)|^2}{\cosh \pi t},
\end{equation}
\[
H = \frac{2}{\pi} \int_0^\infty h(t) \tanh(\pi t) t~dt, ~~~~
H^+(x) = 2i \int_\mathbb{R} J_{2it}(x) \frac{h(t)t}{\cosh \pi t}
~dt, \]
\[
H^-(x) = \frac{4}{\pi} \int_\mathbb{R} K_{2it}(x) \sinh(\pi t) h(t) t~dt, ~~~~~~~~ \text{and}~~~~~~~~
S(a,b;c) = \sum_{d\bar{d} \equiv  1 (\text{mod}\ c)}
e\Big(\frac{da + \bar{d}b}{c}\Big)
\]
is the standard Kloosterman sum.  Above, $J_\nu$ is the $J$-Bessel function.

We let $f$ be a Maass form of type $\nu = (\nu_1, \nu_2)$ for $SL_3(\mathbb{Z})$ (cf. Goldfeld \cite{Gldf}).
Then $f$ has a Whittaker function expansion
$$
f(z) = \sum_{\pm \Gamma^\infty \backslash SL_2(\mathbb{Z})}
\sum_{m_1 = 1}^\infty \sum_{m_2 \not= 0} \frac{A(m_1, m_2)}{m_1 |m_2|}
W_J \Big( M \Big(
\begin{smallmatrix}
\gamma & 0\\
0 &1
\end{smallmatrix}  \Big)
 z, \nu, \psi_{1,1}  \Big),   \label{fwhittaker}
$$
where $W_J$ is the Jacquet-Whittaker function,
$M = \text{diag}(m_1|m_2|, m_1, 1)$, and $\psi_{1,1}$ is a fixed specific generic character
on the abelianization of the standard unipotent upper triangular subgroup of $SL_3(\mathbb{Z})$.
Put
$\alpha = -\nu_1 - 2\nu_2 +1$, 
$\beta = -\nu_1+\nu_2$, 
$\gamma = 2\nu_1 + \nu_2 -1$. 
These are the Langlands parameters at $\infty$ of $f$.  In the usual way, we put
$$
\tilde{\psi}(s) = \int_0^\infty \psi(x) x^{s-1}~dx
$$
to be the Mellin transform of $\psi$ which we assume is smooth
and compactly supported on $(0, \infty)$.

For $k=0,1$ we define
$$
\Psi_k(x) = \int_{\text{Re} s = \sigma} (\pi^3 x)^{-s}
\frac{\Gamma(\frac{1+s+2k+\alpha}{2}) \Gamma(\frac{1+s+2k+\beta}{2}) \Gamma(\frac{1+s+2k+\gamma}{2})}{\Gamma(\frac{-s-\alpha}{2}) \Gamma(\frac{-s-\beta}{2}) \Gamma(\frac{-s-\gamma}{2})}
\tilde{\psi}(-s-k)~ds.
$$
Here $\sigma$ is taken sufficiently large depending on $\alpha, \beta, \gamma$.
We then define, for $k=0,1$,
\begin{equation}\label{Psi0+1}
\Psi_{0,1}^k(x) = \Psi_0(x) + (-1)^k \frac{1}{x \pi^3 i}\Psi_1(x).
\end{equation}
Then the following is a crucial tool, the Voronoi formula for
$GL(3)$.

\begin{lemma}(\cite{MllrSchm})   \label{propVoronoi}
{\it Let $\psi \in C_c^\infty(0,\infty)$.  Let $f$ be a $SL_3(\mathbb{Z})$ Maass form with corresponding
Fourier coefficients $A(m,n)$ as in (\ref{fwhittaker}).  Let $d, \bar{d}, c \in \mathbb{Z}$ with $c\not=0$,
$(d,c) = 1$, and $d\bar{d}\equiv 1 (\!\!\!\mod c)$. Then}
\begin{eqnarray}\label{Voronoidef}
\sum_{n>0} A(m,n) e\Big(\frac{n\bar{d}}{c}\Big)\psi(n)
&=&
\frac{c}{4\pi^{{5}/{2}}i}
\sum_{n_1|cm}
\sum_{n_2>0} \frac{A(n_2, n_1)}{n_1 n_2}
S\Big(md, n_2; \frac{mc}{n_1}\Big)
\Psi_{0,1}^0\Big(\frac{n_1^2 n_2}{c^3 m}\Big)
\\
&&+
\frac{c}{4\pi^{{5}/{2}}i}
\sum_{n_1|cm}
\sum_{n_2>0} \frac{A(n_1, n_2)}{n_1 n_2}
S\Big(md, -n_2; \frac{mc}{n_1}\Big)
\Psi_{0,1}^1\Big(\frac{n_1^2 n_2}{c^3 m}\Big).
\nonumber
\end{eqnarray}
\end{lemma}

To use this formula, asymptotics of $\Psi_0, \Psi_1$ are needed
which were proved in Li \cite{XLi0} and Ren and Ye
\cite{XRenYeY} for $GL(3)$. (For $GL(m)$ see Ren and Ye \cite{XRenYeY2}.)
Since $x^{-1}\Psi_1(x)$ has similar asymptotics to
$\Psi_0$, following \cite{XLi1}, we only deal with $\Psi_0$.
We will use the following Lemma (\cite{XLi0}):

\begin{lemma}\label{Vasymptotics}
Suppose $\psi \in C_c^\infty([X,2X])$.  Then for any fixed integer $K\geq 1$ and $xX \gg 1$ we have   
\[
\Psi_0(x) = 2\pi^3 x i \int_0^\infty \psi(y)
\sum_{j=1}^K \frac{c_j\cos(6\pi (xy)^{{1}/{3}}) + d_j \sin(6\pi (xy)^{{1}/{3}})}{(xy)^{{j}/{3}}}~dy
+O((xX)^{\frac{2-K}{3}}).
\]
Here $c_j$ and $d_j$ are constants depending on the Langlands parameters with $c_1=0$ and $d_1= -{2}/{\sqrt{3\pi}}$.
\end{lemma}

We now assume $f$ is a self-dual Hecke-Maass form for $SL_3(\mathbb{Z})$
of type $(\nu, \nu)$, normalized so that
$A(1,1) = 1$.  The Rankin-Selberg $L$-function of $f$ with itself is then defined by
\[
L(s, f \times f) = \sum_{m \geq 1} \sum_{n \geq 1} \frac{|A(m,n)|^2}{(m^2 n)^s}
\]
for $\text{Re} s$ large.  $L(s, f \times f)$ has meromorphic continuation to the complex plane, with a simple pole
at $s = 1$.  By a standard analytic number theory argument using complex analysis, this gives
$$
\sum_{m^2 n \leq N} |A(m,n)|^2 \ll_f N.
$$
Applying Cauchy-Schwartz, this gives
\begin{equation}
\sum_{n \leq N} |A(m,n)| \ll_f |m|N.    \label{Asumbyparts}
\end{equation}
We will use (\ref{Asumbyparts}) and summation by parts in the estimates below. Here $f$ being self-dual
also means $A(m,n) = A(n,m)$ for all $m, n$.

The Rankin-Selberg $L$-function of $f$ with $u_j$ is
(for $\text{Re} s$ sufficiently large)
$$
L(s, f \times u_j) = \sum_{m \geq 1} \sum_{n \geq 1} \frac{\lambda_j(n) A(m,n)}{(m^2 n)^s}  .
$$
$L(s, f \times u_j)$ can be completed to $\Lambda(s, f \times u_j)$
with six $\Gamma$ factors at $\infty$ (involving the Langlands
parameters of $f$, and $t_j$).

We now need to define the Rankin-Selberg $L$-function of $f$ with the Eisenstein series.
See Li \cite{XLi1} for the definition of $E(z,s)$ and $\eta(n,s)$.
$$
L(s, f \times E) = \sum_{m \geq 1} \sum_{n \geq 1} \frac{\bar{\eta}(n, {1}/{2}+it) A(m,n)}{(m^2 n)^s}.
$$
Following Goldfeld \cite{Gldf}, comparing Euler products, we have
$$
L\Big(\frac{1}{2}, f\times E\Big) =
\Big| L\Big(\frac{1}{2}-it, f\Big)  \Big|^2.
$$

We need to set up the approximate functional equation.  We define
\begin{eqnarray*}
\gamma(s,t)
&=&
\pi^{-3s}
\Gamma\Big(\frac{s-it-\alpha}{2}\Big)
\Gamma\Big(\frac{s-it-\beta}{2}\Big)
\Gamma\Big(\frac{s-it-\gamma}{2}\Big)
\\
&&\times
\Gamma\Big(\frac{s+it-\alpha}{2}\Big)
\Gamma\Big(\frac{s+it-\beta}{2}\Big)
\Gamma\Big(\frac{s+it-\gamma}{2}\Big).
\end{eqnarray*}
Here $\alpha = -3\nu+1$, $\beta = 0$, and $\gamma = 3\nu -1$ are the Langlands parameters at $\infty$ of $f$.
We define $F(u) = \big(\cos(\pi u/A)\big)^{-3A}$ 
for $A$ a positive integer.
For $|\text{Im} t| \leq 1000$ we now define
\begin{equation}\label{V(yt)}
V(y,t) = \frac{1}{2\pi i}
\int_{(1000)} y^{-u} F(u)
\frac{\gamma({1}/{2}+u,t)}{\gamma({1}/{2},t)} ~\frac{du}{u}  .
\end{equation}
By known bounds for the Langlands parameters, this integral converges.  We have the important
approximate functional equation (cf. \cite{XLi1}):

\begin{lemma}
For $f$ a self-dual Maass form of type $(\nu, \nu)$ for $SL_3(\mathbb{Z})$ and $u_j$ a Hecke-Maass form
for $SL_2(\mathbb{Z})$ corresponding to the eigenvalue
${1}/{4}+t_j^2$ in an orthonormal basis, as above,
\begin{equation}
L\Big(\frac{1}{2}, f \times u_j\Big)
=
2\sum_{m \geq 1} \sum_{n \geq 1}
\frac{\lambda_j(n) A(m,n)}{\sqrt{m^2 n}}
V(m^2n, t_j).
\label{approximatefe}
\end{equation}
\end{lemma}

The point of using $V$ in the expansion (\ref{approximatefe}) is that $V$ decays rapidly for
$m^2n \gg |t_j|^{3+\varepsilon}$, and so in an effective way, we can take both sums above to be finite.
For the precise decay rate,   see Lemma 2.3 of Li \cite{XLi1}.
We also have the approximate functional equation for
$L(s, f\times E)$:
\begin{equation}
L\Big(\frac{1}{2}, f \times E\Big)
=
2 \sum_{m \geq 1} \sum_{n \geq 1}
\frac{\eta(n,{1}/{2} + it)A(m,n)}{\sqrt{m^2 n}} V(m^2n, t).
\label{fEapprozfe}
\end{equation}

Following Li \cite{XLi1} we now define
$$
W =
{\mathop{\sum\nolimits'}
\limits_{j}}
e^{-(\frac{t_j-T}{M})^2} \omega_j
L\Big(\frac{1}{2}, f \times u_j\Big)
+ \frac{1}{4\pi} \int_\mathbb{R}
e^{-(\frac{t-T}{M})^2} \omega(t)
\Big| L\Big(\frac{1}{2}-it,f\Big)  \Big|^2  ~dt.
$$
Here $\omega_j$ and $\omega(t)$ are defined in 
\eqref{omegaomega}.
It is known that $\omega_j \gg t_j^{-\varepsilon}$ and $\omega(t) \gg t^{-\varepsilon}$.
See the references in Li \cite{XLi1}. 
It follows that
\[
{\mathop{\sum\nolimits'}
\limits_{j}}
e^{-(\frac{t_j-T}{M})^2}
L\Big(\frac{1}{2}, f \times u_j\Big)
+
\frac{1}{4\pi} \int_\mathbb{R}
e^{-(\frac{t-T}{M})^2}
\Big| L\Big(\frac{1}{2}-it,f\Big)  \Big|^2  ~dt \ll W T^\varepsilon .
\]
Consequently, Theorem \ref{SumIntBdd} will be proved if we show 
$W \ll_{\varepsilon, f} T^{1+\varepsilon}M$.
As Li \cite{XLi1} points out, the function $e^{-(\frac{t-T}{M})^2}$ cannot be used as a test function
in the Kuznetsov trace formula simply
because it is not even.
Following Li \cite{XLi1} we will use the modified function
\begin{equation}
k(t) = e^{-(\frac{t-T}{M})^2} + e^{-(\frac{t+T}{M})^2}   \label{equationk}
\end{equation}
which essentially captures the size of $e^{-(\frac{t-T}{M})^2}$ for $t$ near $T$.
Thus, we define
\begin{equation}
\mathcal{W} =
{\mathop{\sum\nolimits'}
\limits_{j}}
k(t_j) \omega_j L\Big(\frac{1}{2}, f\times u_j\Big)
+ \frac{1}{4\pi}
\int_\mathbb{R} k(t) \omega(t)
\Big|  L\Big(\frac{1}{2} - it, f\Big)  \Big|^2 ~dt.
\label{equationWprime}
\end{equation}

By plugging (\ref{approximatefe}) and (\ref{fEapprozfe}) into $\mathcal{W}$ in (\ref{equationWprime})
we see that we need to analyze
$\mathcal{R}$ which we define by the equation
\begin{eqnarray} \label{equationRVN}
\mathcal{R}
&=&
2{\mathop{\sum\nolimits'}
\limits_{j}}
k(t_j) \omega_j \sum_{m \geq 1} \sum_{n \geq 1} \frac{\lambda_j(n) A(m,n)}{\sqrt{m^2 n}} V(m^2n, t_j)
g\Big(\frac{m^2 n}{N}\Big)  
\\
&&+
\frac{1}{2\pi} \int_\mathbb{R} k(t) \omega(t)
\sum_{m \geq 1}
\sum_{n \geq 1}
\frac{\eta(n,{1}/{2} + it)A(m,n)}{\sqrt{m^2 n}}
V(m^2n, t)
g\Big(\frac{m^2 n}{N}\Big) ~dt.
\nonumber
\end{eqnarray}
Here, for the rest of this article we take
$N = T^{3+\varepsilon}$
and $g$ is a fixed non-negative function with compact support in $[1,2]$.  This is the trick of using a dyadic partition of unity
which is best outlined in Lau, Liu, and Ye \cite{LauLiuYeY}.

Now, we apply the Kuznetsov trace formula (\ref{kuznetsov}) to $\mathcal{R}$ (\ref{equationRVN}).
Consequently, we write
\begin{equation}\label{Rdefine}
\mathcal{R} = \mathcal{D} + \mathcal{R}^+  + \mathcal{R}^-;
\end{equation}
\begin{equation}\label{Ddefine}
\mathcal{D} = \sum_{m \geq 1} \sum_{n \geq 1} \frac{A(m,n)}{\sqrt{m^2 n}}
g\Big(\frac{m^2 n}{N}\Big) \delta(n,1) H_{m,n};
\end{equation}
$$
H_{m,n}  =
\frac2\pi
\int_\mathbb{R} k(t) V(m^2 n, t) \tanh (\pi t) t~dt ;
$$
\begin{equation}\label{equationRplus}
\mathcal{R}^+  
=
\sum_{m \geq 1}  \sum_{n \geq 1}  \frac{A(m,n)}{\sqrt{m^2 n}} g\Big(\frac{m^2 n}{N}\Big)
\sum_{c > 0} \frac{S(n,1;c)}{c} H_{m,n}^+ \Big(  \frac{4\pi \sqrt{n}}{c}   \Big);
\end{equation}
\begin{equation}
H_{m,n}^+(x)  
=
2i\int_\mathbb{R} J_{2it} (x) \frac{k(t) V(m^2 n, t)t}{\cosh (\pi t)} ~dt;
\label{Hmnplus}
\end{equation}
\begin{equation}\label{equationRminus}
\mathcal{R}^- = \sum_{m \geq 1}  \sum_{n \geq 1}  \frac{A(m,n)}{\sqrt{m^2 n}}
g\Big(\frac{m^2 n}{N}\Big)
\sum_{c > 0} \frac{S(n,-1;c)}{c}
H_{m,n}^- \Big(  \frac{4\pi \sqrt{n}}{c}   \Big);
\end{equation}
\begin{equation}\label{Hmnminus}
H_{m,n}^-(x)  = \frac{4}{\pi} \int_\mathbb{R} K_{2it} (x) \sinh(\pi t) k(t) V(m^2 n, t)t ~ dt.
\end{equation}

By the estimates in Section 3 of Li \cite{XLi1}, we see easily that $\mathcal{D}$ in \eqref{Ddefine} is negligible
for any $M$ with
$T^\varepsilon \leq M \leq T^{1-\varepsilon}$
and we leave the details for the reader. In the next
two section we will estimate $\mathcal{R}^+$ in 
\eqref{equationRplus} and $\mathcal{R}^-$ in 
\eqref{equationRminus}.

\section{Estimates for the $J$-Bessel function terms}

In this section we provide estimates for $\mathcal{R}^+$ in (\ref{equationRplus}).
In this section and the next, we show estimates under the assumption $T^{1/3 + 2\varepsilon} \leq M \leq T^{1/2}$.
Following Li \cite{XLi1} we define the parameters
\begin{equation}
C_1 = T^{100}, ~~~~\text{and}~~~~C_2 = \frac{\sqrt{N}}{T^{1-\varepsilon}M},    \label{Cconstants}
\end{equation}
and we split $\mathcal{R}^+ = \mathcal{R}_1^+ + \mathcal{R}_2^+ + \mathcal{R}_3^+$ with
\begin{equation}
\mathcal{R}_1^+
= \sum_{m \geq 1}
\sum_{n \geq 1}  \frac{A(m,n)}{\sqrt{m^2 n}}
g\Big(\frac{m^2 n}{N}\Big)
\sum_{  c \geq C_1/m} \frac{S(n,1;c)}{c}
H_{m,n}^+ \Big(  \frac{4\pi \sqrt{n}}{c}   \Big),
\label{eqnR1+}
\end{equation}
\begin{equation}
\mathcal{R}_2^+
= \sum_{m \geq 1}
\sum_{n \geq 1}  \frac{A(m,n)}{\sqrt{m^2 n}}
g\Big(\frac{m^2 n}{N}\Big)
\sum_{C_2/m \leq c \leq C_1/m} \frac{S(n,1;c)}{c}
H_{m,n}^+ \Big(  \frac{4\pi \sqrt{n}}{c}   \Big),
\label{eqnR2+}
\end{equation}
\begin{equation}
\mathcal{R}_3^+
= \sum_{m \geq 1}
\sum_{n \geq 1}  \frac{A(m,n)}{\sqrt{m^2 n}}
g\Big(\frac{m^2 n}{N}\Big)
\sum_{c \leq C_2/m} \frac{S(n,1;c)}{c}
H_{m,n}^+ \Big(  \frac{4\pi \sqrt{n}}{c}   \Big).
\label{eqnR3+}
\end{equation}

For $\mathcal{R}_1^+$ in \eqref{eqnR1+}, Li \cite{XLi1} shifts the integral defining $H_{m,n}^+$ (see (\ref{Hmnplus})),
and uses an integral representation
of the $J$-Bessel function and Stirling's formula to conclude
\begin{equation}\label{Hmnplusestimate}
H_{m,n}^+(x) \ll x^{\frac{3}{4}}
T^{\frac{3}{8}} (m^2 n )^{-\frac{3}{8}} T^{1+\varepsilon}M.    
\end{equation}
Consequently \eqref{eqnR1+} is bounded
\begin{equation}\label{R1+<<}
\mathcal{R}_1^+ \ll T^{\frac{11}{8}+\varepsilon}M
\sum_{m \leq \sqrt{2N}}
\sum_{n \leq {2N}/{m^2}}
\frac{|A(m,n)|}{m\sqrt{n}}
\sum_{c \geq C_1/m}
\frac{|S(n,1;c)|}{c}
\Big(\frac{\sqrt{n}}{c}\Big)^{\frac{3}{4}} \cdot (m^2 n)^{-\frac{3}{8}} .
\end{equation}
Using Weil's bound for $S(n,1;c)$, we see
\begin{equation}\label{Ssum<<}
\sum_{c \geq C_1/m} \frac{|S(n,1;c)|}{c^{\frac{7}{4}}} \ll \sum_{c \geq C_1/m} \frac{c^{\frac{1}{2}+\varepsilon}}{c^{\frac{7}{4}}}
\ll \Big(\frac{C_1}m\Big)^{-\frac{1}{4}+\varepsilon}.
\end{equation}
By (\ref{Asumbyparts}) and summation by parts, we have
\begin{equation}\label{Asum<<}
\sum_{n \leq {2N}/{m^2}} \frac{|A(m,n)|}{\sqrt{n}}
\ll m\Big(\frac{N}{m^2}\Big)^{\frac{1}{2}}.
\end{equation}
Inserting \eqref{Ssum<<} and \eqref{Asum<<} into \eqref{R1+<<} 
we get 
\begin{equation}\label{R1+Bdd}
\mathcal{R}_1^+ \ll T^{\frac{11}{8}+\varepsilon}M N^{\frac{1}{2}} C_1^{-\frac{1}{4}} \sum_{m \leq \sqrt{2N}} \frac{1}{m^{\frac{3}{2}}}.
\end{equation}
Plugging in $C_1 = T^{100}$ from \eqref{Cconstants}, 
$N = T^{3+\varepsilon}$ and noticing the sum on $m$ in 
\eqref{R1+Bdd} converges, we have
$\mathcal{R}_1^+ \ll 1$ for any $M$ with $T^\varepsilon \leq M \leq T^{1-\varepsilon}$.

We now deal with $\mathcal{R}_2^+$ in \eqref{eqnR2+}.  We do not wish to reproduce all the estimates in Li \cite{XLi1} so we will
summarize. As used in Liu and Ye \cite{LiuYeY}
\cite{pet-kuz} and Li \cite{XLi1}
we need an integral representation for
\[
\frac{J_{2it}(x) - J_{-2it}(x)}{\cosh (\pi t)}
\]
from 1.13(69) of \cite{BtmnT}, vol.1, p.59.
Using integration by parts, a change of variables, and the fact that $k(t)$ (recall (\ref{equationk})) is a
Schwartz function, we define
\[
W_{m,n}(x) = T\int_\mathbb{R} \widehat{k^*} (\zeta)
\cos \Big(  x\cosh \Big(\frac{\zeta \pi}{M}\Big)  \Big)
e\Big(  -\frac{T\zeta}{M}  \Big) ~d\zeta.
\]
Here
\[
k^*(t) = e^{-t^2} V(m^2n , tM+T)
\]
is a Schwartz function, and $\widehat{k^*}$ is its Fourier
transform. We remark that derivatives of $k^*(t)$ are $\ll1$.
In fact, by \eqref{V(yt)}
$\frac{\partial^\ell}{\partial t^\ell}V(y,tM+T)$ can be
expressed in terms of derivatives of $\gamma(s,tM+T)$ and hence
in terms of $\frac d{dz}\log\Gamma(z)=:\psi(z)$ and
$\psi^{(\ell)}(z)$ (Bateman \cite{BtmnH} p.15, 1.7(1), and
p.45, 1.16(9)). By their
asymptotic expansions in \cite{BtmnH}, p.47, 1.18(7), and
p.48, 1.18(9), we can see
$$
\frac{\partial^\ell}{\partial t^\ell}V(y,tM+T)
\ll \Big(\frac MT\Big)^\ell.
$$

We define
\begin{equation}
W_{m,n}^*(x)
= T \int_\mathbb{R}
\widehat{k^*} (\zeta) e\Big( -\frac{T\zeta}{M} - \frac{x}{2\pi} \cosh\Big(\frac{\zeta \pi}{M}\Big)  \Big)
~d\zeta,    \label{Wmnstar}
\end{equation}
so that
\[
W_{m,n}(x) = \frac{W_{m,n}^*(x)+W_{m,n}^*(-x)}{2}.
\]
The upshot here is that up to a lower order term (which can be handled in a similar way)
and a negligible amount, we have
$H_{m,n}^+(x) = 4W_{m,n}(x)$.

The contribution to the integral in (\ref{Wmnstar}) from $|\zeta|\geq T^\varepsilon$ is a negligible amount, so
in what follows we can assume
$|\zeta| \leq T^\varepsilon$.  The phase $\phi(\zeta)$ in the exponential (\ref{Wmnstar}) is
\[
2\pi\phi(\zeta) = -\frac{T\zeta}{M} - \frac{x}{2\pi} \cosh\Big(\frac{\zeta \pi}{M}\Big).
\]
Looking at $\phi'(\zeta)$, we see $W_{m,n}^*(x)$ is negligible for $|x|\leq T^{1-\varepsilon}M$.  So in what follows
we assume $T^{1-\varepsilon}M \leq |x| \leq T^2$. 
Using a Taylor expansion in $\zeta$
(within the exponential) of
$$
e\Big( -\frac{T\zeta}{M} - \frac{x}{2\pi}
\cosh\Big(\frac{\zeta \pi}{M}\Big)  \Big)
$$
in (\ref{Wmnstar}), using the
Fourier transform of a Gaussian, using Parseval's Theorem, completing the square, and working out many estimates, 
Lau, Liu and Ye (Lemma 5.1 of \cite{LauLiuYeY}) and
Li (Proposition 4.1 of \cite{XLi1}) proved similar propositions, estimating $W_{m,n}^*(x)$ by a finite series
involving derivatives of $ \widehat{k^*}$,
 based on ideas in Sarnak \cite{Srnk}.  For our purposes we can modify the proof of Proposition 4.1 of \cite{XLi1}.

\begin{lemma}
 \label{propestimate}{\it
1)  For $|x| \leq T^{1-\varepsilon}M$ we have 
$W_{m,n}^*(x) \ll_{\varepsilon, A}   T^{-A}$.

2) For $T^{1-\varepsilon}M \leq |x| \leq T^2$, with
$T^{1/3+2\varepsilon} \leq M \leq T^{{1}/{2}}$ and $L_1, L_2 \geq 1$,
\begin{eqnarray}
W_{m,n}^*(x)
&=&
\frac{TM}{\sqrt{|x|}}
e\Big( -\frac{x}{2\pi} + \frac{T^2}{\pi x} \Big)
\sum_{l = 0}^{L_1}
\sum_{0 \leq l_1 \leq 2l}
\sum_{\frac{l_1}{4}\leq l_2 \leq L_2}
c_{l, l_1, l_2}
\frac{M^{2l - l_1} T^{4l_2-l_1}}{x^{l+3l_2-l_1}}
\label{Wmnestimate}
\\
&&\times
\Big[    \widehat{k^*}^{(2l-l_1)}
\Big(-\frac{2MT}{\pi x}\Big)
- \frac{\pi^6 i x}{6! M^6} (y^6 \widehat{k^*}(y))^{(2l-l_1)}
\nonumber
\\
&&+
\frac{\pi^{12} i^2 x^2}{2! (6!)^2 M^{12}} (y^{12} \widehat{k^*}(y))^{(2l-l_1)}
\Big(-\frac{2MT}{\pi x}\Big)
\Big]
\nonumber
\\
&&+
O\Big( \frac{TM}{\sqrt{|x|}}
\Big(\frac{T^4}{|x|^3}\Big)^{L_2+1}
+T\Big(\frac{M}{\sqrt{|x|}}\Big)^{2L_1+3}
+ \frac{T|x|^3}{M^{18}}
\Big),
\nonumber
\end{eqnarray}
where $c_{l, l_1, l_2} $ are constants depending only on the indices.}
\end{lemma}

Note Part $1)$ is valid for $T^\varepsilon \leq M \leq T^{1-\varepsilon}$, and Part $2)$
is  valid for $T^{{1}/{3} + \varepsilon} \leq M \leq \sqrt{T}$ with the assumption of 
$T^{1-\varepsilon}M \leq |x| \leq T^2$.
With our assumption 
$T^{{1}/{3} + 2\varepsilon} \leq M \leq \sqrt{T}$ on $M$, 
to acquire the desired decay rate of the
\begin{equation}\label{OTMpower}
O\Big( \frac{TM}{\sqrt{|x|}}
\Big(\frac{T^4}{|x|^3}\Big)^{L_2+1} \Big)
\end{equation}
term, $L_2$ could depend on $\varepsilon$.
From $1)$ of Lemma \ref{propestimate} and 
\eqref{Hmnplusestimate} we see $\mathcal{R}_2^+$ is negligible.
The extra term in the brackets in (\ref{Wmnestimate}), as compared to  \cite{XLi1},
comes from a degree $2$ Taylor expansion in $x$
(with remainder) of
$e(-\pi^6 i x \zeta^6/(2\cdot 6! M^6))$.

In the rest of this section, we estimate $\mathcal{R}_3^+$ 
as in \eqref{eqnR3+}.  By choosing $L_1, L_2$
large enough
(possibly depending on $\varepsilon$) in (\ref{Wmnestimate}) the contribution to $\mathcal{R}_3^+$
from the first two error terms in (\ref{Wmnestimate}) can be made as small as desired.
We need to estimate the contribution from the last error term in (\ref{Wmnestimate}).
By the support of $g$ we may assume 
$x^2 =16\pi^2 n/c^2 \ll N = T^{3 + \varepsilon}$.
By our assumptions on $M$ and $T$ we then have 
$T|x|^3/M^{18}  \ll T|x|/M^9$.
Plugging in $x= 4\pi \sqrt{n}/c$ into $T|x|/M^9$, we estimate this error term contribution to
$\mathcal{R}_3^+$ in (\ref{eqnR3+}), using 
\eqref{Hmnplusestimate}, Weil's bound for the
Kloosterman sum, and the compact support of $g$.  This error can be seen to be bounded by
$O( TN/M^9)$ which is smaller than $O( T^{1+\varepsilon}M)$ by a power of $T$
with our assumption
$T^{{1}/{3} + 2\varepsilon} \leq M \leq \sqrt{T}$.  In the finite series 
(\ref{Wmnestimate}) with our assumptions we also have
$M^{2l - l_1}T^{4l_2-l_1}x^{l_1-l-3l_3} \ll 1$. 
All the terms in (\ref{Wmnestimate}) are similar, and can be estimated in a similar way, so we will only work
with the first term.
Following Li \cite{XLi1} we define
\begin{eqnarray}   \label{mathcalR3}
\widetilde{\mathcal{R}}_3^+
&=&
\frac{i(i+1) MT}{\sqrt{2} \pi}
\sum_{m\geq 1} \sum_{n\geq 1}
\frac{A(m,n)}{mn^{{3}/{4}}}
g\Big(\frac{m^2 n}{N}\Big)
\\
&&\times
\sum_{c \leq C_2/m} \frac{S(n,1;c)}{\sqrt{c}} e\Big(
\frac{2\sqrt{n}}{c} -\frac{T^2c}{4\pi^2\sqrt{n}}  \Big)
\widehat{k^*} \Big(  \frac{MTc}{2\pi^2 \sqrt{n}}  \Big).
\nonumber
\end{eqnarray}

Li \cite{XLi1} points out here, that even with Weil's bound for $S(n,1;c)$ simple estimates for
$\widetilde{\mathcal{R}}_3^+$ are too large.  So we expand the Kloosterman sum $S(n,1;c)$ and use
the Voronoi formula (Lemma \ref{propVoronoi})
with
\begin{equation}\label{psiy}
\psi(y)
= y^{-\frac{3}{4}}
g\Big(\frac{m^2 y}{N}\Big)
e\Big(
\frac{2\sqrt{y}}{c} -\frac{T^2c}{4\pi^2\sqrt{y}}  \Big) \widehat{k^*} \Big(  \frac{MTc}{2\pi^2 \sqrt{y}}  \Big).
\end{equation}
We get
\begin{equation}\label{R3+short}
\widetilde{\mathcal{R}}_3^+
=
\frac{(i-1) MT}{\sqrt{2} \pi}
\sum_{m\geq 1}\frac1m
\sum_{c\leq C_2/m}\frac1{\sqrt c}
{\mathop{\sum\nolimits^*}
\limits_{d\,(\!\!\!\!\!\mod c)}}e\Big(\frac dc\Big)
\sum_{n\geq1}
A(m,n)e\Big(\frac{n\bar d}c\Big)\psi(n),
\end{equation}
where the innermost sum in \eqref{R3+short} will be replaced 
by the right hand side of \eqref{Voronoidef}. 

From the function $g(m^2y/N)$ in \eqref{psiy} we can see that 
$X=N/m^2$. Recall $x = {n_2 n_1^2}/(c^3 m)$ from Lemma 
\ref{propVoronoi}. Then by $c \leq C_2/m$ 
$$
xX
=
\frac{n_2n_1^2N}{c^3m^3}
\geq
\frac{n_2n_1^2N}{C_2^3}
=
\frac{n_2n_1^2T^{3-3\varepsilon}M^3}{\sqrt N}
\geq
n_2n_1^2T^{3/2-3\varepsilon}M^3
\gg
1.
$$
Consequently we can apply Lemma \ref{Vasymptotics} to 
\eqref{R3+short} with \eqref{Voronoidef} 
to get 
\begin{equation}\label{Psi0asymptotics}
\Psi_0(x) = \pi^3 d_1 x^{{2}/{3}}   \int_0^\infty e(u_1(y))a(y)~dy -
\pi^3 d_1 x^{{2}/{3}}   \int_0^\infty e(u_2(y))a(y)~dy  
\end{equation}
with
\begin{equation}\label{u1u2}
u_1(y) = \frac{2\sqrt{y}}{c} +3(xy)^{{1}/{3}}, ~~~~u_2(y) = \frac{2\sqrt{y}}{c} -3(xy)^{{1}/{3}}
\end{equation}
and
\begin{equation}   \label{equationay}
a(y) =
g\Big(\frac{m^2 y}{N}\Big)
\widehat{k^*}\Big(\frac{MTc}{2 \pi^2 \sqrt{y}}\Big)
e\Big(\frac{-T^2c}{4 \pi^2 \sqrt{y}}\Big)
y ^{-{13}/{12}}.
\end{equation}
Note that $u_1$ has no stationary points; indeed  simple calculus estimates
give the first integral in (\ref{Psi0asymptotics}) a negligible contribution to
$\widetilde{\mathcal{R}}_3^+$.

The second integral in (\ref{Psi0asymptotics}) requires more analysis. As in \cite{XLi1}, p.319, if 
$x\geq2\sqrt{N}/(c^3 m)$ or $x\leq2\sqrt{N}/(3c^3 m)$, 
then $u_2'(y)$ will be effectively
bounded away from zero, making the integral negligible by multiple integration by parts.
Thus we assume the contrary in what follows, namely
\begin{equation}
\frac{2\sqrt{N}}{3n_1^2} \leq n_2 \leq \frac{\sqrt{N}}{n_1^2}.     \label{equationn1n2}
\end{equation}
We have
\begin{equation}\label{eaintegral}
\int_0^\infty e(u_2(y)) a(y) ~dy = \int_{\frac{1}{4}x^2 c^6}^{\frac{9}{2}x^2 c^6} e(u_2(y)) a(y) ~dy.
\end{equation}
We explain the limits of integration.
The compact support of the integral on the right side of equation (\ref{eaintegral}) follows from the compact support
of $g$, and so that of $a$. Further, recall
$x = {n_2 n_1^2}/(c^3 m)$.  As Li \cite{XLi1} points out, the stationary phase point of the integral in (\ref{eaintegral}) is
at $y_0 = x^2 c^6$.  The constants $1/4$ and $9/2$ in the limits of this integral give a segment that
the support of $a$ is contained in, since $g \in C_c^\infty([1,2])$.
In (\ref{equationay}), from the support of $g$, and since $\widehat{k^*}$ is a Schwartz function, we can assume
\[
\frac{N}{m^2} \leq y \leq\frac{2N}{m^2}  ~~\text{and}~~ \frac{MTc}{2\pi^2 \sqrt{y}} \ll T^\varepsilon.
\]
Using this information, simple calculus estimates give us
\begin{equation}   \label{equationu2a}
u_2^{(r)} (y) \ll T_1 M_1^{-r}  ~~\text{for}~ r = 1, 2, \ldots, 2n_0 +3
\end{equation}
and
\begin{equation}   \label{N-equationu2a}
a^{(r)}(y) \ll U_1 N_1^{-r}~~\text{for} ~r = 0, 1, 2, \ldots, 2n_0 +1
\end{equation}
for $y$ in the segment. Here $n_0 \in \mathbb{N}$ will be chosen in terms of $\varepsilon_0$ later, and
\begin{equation}\label{MTNU}
M_1 = \frac{N}{m^2}, ~~T_1 = \frac{\sqrt{N}}{cm}, ~~N_1 = \frac{N^{3/2}}{T^2 c m^3},
~~U_1 = \Big(\frac{N}{m^2}\Big)^{-13/12} .     
\end{equation}
Further,    $u_2^{(2)}(y) \gg T_1 M_1^{-2}$ for $y \in [\frac{1}{4}x^2 c^6,  \frac{9}{2}x^2 c^6]$.
The condition $N_1 \geq {M_1}/{\sqrt{T_1}}$ is then consistent with our assumption
$c \leq C_2/m$ when $M \geq T^{{1}/{3}+2\varepsilon}$.

Then, all assumptions \eqref{M-N} and \eqref{UfLfUg} 
are satisfied for parameters in \eqref{MTNU}, and we apply Proposition \ref{WSPI}   
(where we take $n=n_0$). Or, one may use 
Blomer, Khan, and Young's version in \cite{BlmKhnYng}. 
The main term of the integral in (\ref{eaintegral}) is
\begin{equation}\label{mainterm}
\frac{e(u_2(y_0)\pm{1}/{8})}{\sqrt{|u_2^{''}(y_0)|}}
\Big( a(y_0) + \sum_{j=1}^{n_0}\varpi_{2j}\frac{(-1)^{j}(2j-1)!!}{(4\pi i\lambda_2)^j}    
\Big),
\end{equation}
where $\varpi_{2j}$ are defined above and
$\lambda_2 = {u''_2(y_0)}/{2}$.
Notice we have used
$\gamma - \alpha  \asymp \beta - \gamma \asymp  M_1$,
with $\alpha = \frac{1}{4}x^2 c^6$, $\beta = \frac{9}{2}x^2 c^6$ and $\gamma = y_0 = x^2 c^6$.
To save time in estimates, notice there are no boundary terms here.
This is due to the compact support of $a$, with itself and all of its derivatives zero at
$\frac{1}{4}x^2 c^6$ and $\frac{9}{2}x^2 c^6$.
The sum of the four error terms in Proposition \ref{WSPI} can be simplified to
\begin{equation}\label{estimateOO}
O \Big( \frac{U_1 M_1^{2n_0 +2}  }{T_1^{n_0 + 1} N_1^{2n_0 +1}}  \Big).  
\end{equation}
This estimate uses the current assumptions on $c$ and $m$, and the size of $N$ compared to $T$.
Note that $M_1 \gg N_1$.

We need to estimate this error term, as well as error terms coming from the $\varpi_{2j}$ terms which will be very
similar.
First we need a nifty estimate from
Li \cite{XLi1}.
Using the basic definitions, as Li points out (equation (4.22)  of \cite{XLi1})
\begin{equation}
{\mathop{\sum\nolimits^*}
\limits_{0 \leq d \leq c }}
e\Big(\frac{d}{c}\Big) S(md, n_2; mcn_1^{-1})
=
{\mathop{\sum\nolimits^*}
\limits_{u\,(\text{mod}\, mcn_1^{-1}) } }
S(0, 1+un_1 ; c)
e\Big(\frac{n_2 \bar{u}}{mcn_1^{-1}}\Big) .   \label{Ramanujan}
\end{equation}
Here $u \bar{u} \equiv 1 (\!\!\!\mod mcn_1^{-1}) $ and
\[
S(0, a; c) =
{\mathop{\sum\nolimits^*}
\limits_{v(\!\!\!\!\! \mod c) }}
e\Big(\frac{av}{c}\Big)
\]
is the Ramanujan sum, which is $\ll (a,c)$.
Then (\ref{Ramanujan}) is bounded by
\begin{equation}\label{<<sumd1}
\ll
{\mathop{\sum\nolimits^*}
\limits_{u\,(\text{mod}\, mcn_1^{-1}) } }   (1+un_1, c)
~=~ \sum_{d|c} d
{\mathop{\sum\nolimits}
\limits_{ \stackrel{u\,(\text{mod}\, mcn_1^{-1})}{ (1+un_1, c) = d } } }
1
 ~\ll~
\sum_{d|c} d
{\mathop{\sum\nolimits}
\limits_{ \stackrel{u\,(\text{mod}\, mcn_1^{-1})}{ un_1 \equiv -1 (\!\!\!\mod d) } } }  1.
\end{equation}
Now $(n_1,d)=1$ and so $\bar n_1$ exists $(\!\!\!\! \mod d)$. 
Thus the last inner sum in \eqref{<<sumd1} is over all $u$ with
$0 \leq u < mcn_1^{-1}$ and $u \equiv -\bar n_1 (\!\!\!\mod d)$.  The number of such terms is clearly $\asymp mc/(dn_1)$.
Plugging this into \eqref{<<sumd1} we see that 
(\ref{Ramanujan}) is bounded by 
\begin{equation}\label{importantsavings}
{\mathop{\sum\nolimits^*}
\limits_{0 \leq d \leq c }}
e\Big(\frac{d}{c}\Big) S(md, n_2; mcn_1^{-1})
\ll \frac{mc}{n_1} \sum_{d|c} 1 
\ll \frac{mc^{1+\varepsilon}}{n_1} .
\end{equation}

Now let us turn back to \eqref{R3+short} with \eqref{Voronoidef} 
and \eqref{Psi0+1}. As we pointed out before, we will only 
consider the contribution from $\Psi_0(x)$ 
for $x=n_2n_1^2/(c^3m)$. In other words,
\begin{equation}\label{R3+<<Kl}
\widetilde{\mathcal{R}}_3^+
\ll  MT 
\sum_{m \leq C_2}
\frac1m
\sum_{c \leq {C_2}/{m}}
c^{1/2}
\sum_{n_1 | cm}
\sum_{n_2 >0}
\frac{|A(n_2, n_1)|}{n_1n_2}
\Big|\Psi_0\Big(\frac{n_2n_1^2}{c^3m}\Big)\Big|
\ \Big|
{\mathop{\sum\nolimits^*}
\limits_{0 \leq d \leq c }}
e\Big(\frac{d}{c}\Big) S(md, n_2; mcn_1^{-1})
\Big|.
\end{equation}
We know we need actually consider the 
contribution from the second term in \eqref{Psi0asymptotics}. 
Using \eqref{importantsavings}, \eqref{R3+<<Kl} can be 
reduced to  
\begin{equation}\label{R3+<<int}
\widetilde{\mathcal{R}}_3^+
\ll  MT 
\sum_{m \leq C_2}
m^{-2/3}
\sum_{c \leq {C_2}/{m}}
c^{-1/2+\varepsilon}
\sum_{n_1 | cm}
n_1^{-2/3}
\sum_{n_2\asymp\sqrt N/n_1^2}
\frac{|A(n_2, n_1)|}{n_2^{1/3}}
\int_0^\infty
e(u_2(y))a(y)dy.
\end{equation}
The following Lemma is specific to the estimation of 
(\ref{R3+<<int}).

\begin{lemma}   \label{lemmacTNm}
Assume $\alpha\geq -1/2$ and $\delta - \alpha \geq 1/6$. 
Suppose we have a term bounded by 
$O(c^\alpha T^\beta N^\gamma m^\delta) $ with
specific numbers $\alpha, \beta, \delta, and \gamma$ for the  integral in \eqref{R3+<<int}.
Then the contribution of this term to 
$\widetilde{\mathcal{R}}_3^+$ is
\[
\ll  
M^{2/3 - \delta - 2\varepsilon} 
T^{13/6 + \beta  + 3\gamma + \delta/2 + \varepsilon_1}
\]
where $\varepsilon$ is arbitrarily small from 
\eqref{importantsavings}, and 
$\varepsilon_1 = \varepsilon
(11/6+ 3\delta/2 + \gamma) + 3\varepsilon^2$.
\end{lemma}

\begin{proof}
By \eqref{R3+<<int} the contribution of 
$O(c^\alpha T^\beta N^\gamma m^\delta) $  
to $\widetilde{\mathcal{R}}_3^+$ is 
\begin{equation}\label{<<4sums}
\ll  MT \sum_{m \leq C_2}
m^{-2/3}
\sum_{c \leq {C_2}/{m}}
c^{-{1}/{2}+\varepsilon}
\sum_{n_1 | cm}
n_1^{-2/{3}}
\sum_{n_2 \asymp {\sqrt{N}}/{n_1^2}}
\frac{|A(n_1, n_2)|}{n_2^{{1}/{3}}}
c^\alpha T^\beta N^\gamma m^\delta.
\end{equation}
Note that the innermost sum in \eqref{<<4sums} is over
(\ref{equationn1n2}). Also note Li \cite{XLi1} seems 
to have used the estimate $(mc)^{1+\varepsilon}$ instead of the estimate
${mc^{1+\varepsilon}}/{n_1}$ from (\ref{importantsavings}). 
Since the sum on $n_1$ is a divisor sum, this is not an issue here.
Using the estimates for $|A(n_1, n_2)|$ (see (\ref{Asumbyparts})), and partial summation one has
$$
\sum_{n_2 \asymp {\sqrt{N}}/{n_1^2}}
\frac{|A(n_1, n_2)|}{n_2^{{1}/{3}}}
\ll n_1
\Big(\frac{\sqrt{N}}{n_1^2}\Big)^{{2}/{3}}.
$$
Since the number of divisors of $cm$ is $\ll (cm)^\varepsilon$ this simplifies the contribution to \eqref{<<4sums} to
\begin{equation}\label{<<2sums}
\ll ~MT^{1+\beta}N^{{1}/{3}+ \gamma}
\sum_{m \leq C_2}
m^{-{2}/{3}+\varepsilon + \delta}
 \sum_{c \leq {C_2}/{m}}
c^{-{1}/{2}+2\varepsilon +\alpha}.
\end{equation}
From a calculus estimate, we have
$$
 \sum_{c \leq {C_2}/{m}}
c^{-{1}/{2}+2\varepsilon +\alpha}
\ll
\Big(\frac{C_2}m\Big)^{{1}/{2}+2\varepsilon +\alpha},
$$
because $\alpha\geq-1/2$ and $m\leq C_2$. 
Plugging this into \eqref{<<2sums}, and using
$C_2 = \sqrt{N}/(T^{1-\varepsilon}M)$
we have
\begin{equation}\label{<<1sum}
\ll
MT^{1+\beta}N^{1/3+ \gamma} 
\Big( \frac{\sqrt{N}}{T^{1-\varepsilon}M} \Big)
^{1/2+2\varepsilon +\alpha}
\sum_{m \leq C_2}  
m^{-7/6 + \delta - \alpha-\varepsilon}  .
\end{equation}
Now, since $\delta-\alpha\geq1/6$, we have
\begin{equation}\label{msum<<}
\sum_{m \leq C_2}  
m^{-7/6 + \delta-\alpha-\varepsilon}  
\ll
C_2^{-1/6 +\delta - \alpha-\varepsilon}+1
\ll
C_2^{-1/6 +\delta - \alpha},
\end{equation}
because 
$C_2=\sqrt N/(T^{1-\varepsilon}M)
=T^{1/2+\varepsilon}/M\geq T^\varepsilon$. 
Inserting \eqref{msum<<} into \eqref{<<1sum}, we see 
\eqref{<<4sums} is bounded by
$$
\ll
M^{2/3-\delta -2\varepsilon} 
T^{2/3+\beta-\delta+\varepsilon(\delta-5/3+2\varepsilon)}
N^{1/2+ \gamma + \delta/2 + \varepsilon}.
$$
Now plugging in $N= T^{3+\varepsilon}$ gives our Lemma.
\end{proof}

Now let us turn back to the error term (\ref{estimateOO}). 
By \eqref{MTNU}, \eqref{estimateOO} can be written as 
\begin{equation}\label{OcTNm}
O  \Big(   c^{3n_0 +2} T^{4n_0 + 2  } N^{-\frac{3}{2}n_0 - \frac{13}{12}}
m^{3n_0 + \frac{13}{6}}   \Big).
\end{equation}
Since $(3n_0+13/6)-(3n_0+2)=1/6$, we may apply Lemma 
\ref{lemmacTNm} to \eqref{OcTNm} and get its contribution to 
$\widetilde{\mathcal{R}}_3^+$ as
\begin{equation}\label{MTpower}
O(M^{-3n_0-3/2-2\varepsilon}
T^{n_0+2+\varepsilon_1}),
\end{equation}
where $\varepsilon>0$ is arbitrarily small as in 
\eqref{importantsavings} and 
$\varepsilon_1=\varepsilon(3n_0+4)+3\varepsilon^2$. 
For any $\varepsilon_0>0$ arbitrarily small, we want to 
make \eqref{MTpower} $\ll T^{1+\varepsilon_0}M$. This can be 
done if
\begin{equation}\label{n0estimate}
M \geq 
T^{\frac{n_0 + 1 +\varepsilon_1 -\varepsilon_0}
{3n_0 + 5/2 +  \varepsilon}}.
\end{equation}
We will choose $n_0$ later depending on $\varepsilon_0$. 
Notice that if $n_0 = 1/2$, we pick up the $3/8$ constant 
of Li \cite{XLi1} from \eqref{n0estimate}. This 
concludes the estimation of contribution of error terms 
\eqref{estimateOO} 
in Proposition \ref{WSPI} to $\widetilde R_3^+$. 

We now need to deal with the $\varpi_{2j}$ terms in (\ref{mainterm}) and their contribution to $\widetilde R_3^+$. 
Recall the expression for
$\varpi_{2j}$ in (\ref{varpi-5}).  Here we take $2 \leq 2j \leq 2n_0$.
One can see from (\ref{varpi-5}) that the main term from $\varpi_{2j}$ is $a^{(2j)}(y_0)$.
(Here $a(y)$ given in (\ref{equationay}) and $u_2(y)$ in 
(\ref{u1u2}) take the place of $g$ and $f$ in Proposition 
\ref{WSPI}.  Further $y_0$ takes the place of $\gamma$.)
Using the estimates in (\ref{equationu2a}) and (\ref{N-equationu2a}) along with $|u_2''(y_0)|\gg T_1/ M_1^2$, and 
along with our
current assumptions on $c$ and $m$ in \eqref{mathcalR3} 
we have
\begin{equation}\label{varpi-a}
\varpi_{2j} - a^{(2j)}(y_0) 
= 
O \Big(    \frac{U_1}{M_1 N_1^{2j-1}}
\Big)
\end{equation}
The constant ultimately depends on $n_0$ and we have used $M_1 \gg N_1$.  To estimate the contribution of this error term 
\eqref{varpi-a} to $\widetilde{\mathcal{R}}_3^+$, we must divide 
by $\lambda_2^{j+ \frac{1}{2}}$ and sum over $j$.
(See (\ref{mainterm}).)
Since $y_0 \asymp N/m^2$, we have
 $\lambda_2 \asymp m^3N^{-3/2}/c$. 
We have then that this contribution is
\[
\ll \Big(  \frac{N}{m^2} \Big)^{-\frac{25}{12}}    \Big( \frac{T^2 c m^3}{N^{\frac{3}{2}}}  \Big)^{2j-1}
 \Big(  \frac{cN^{\frac{3}{2}}}{m^3} \Big)^{j + \frac{1}{2}} =
 O\Big(  c^{3j - \frac{1}{2}} T^{4j-2} N^{-\frac{3}{2}j + \frac{1}{6}}  m^{3j - \frac{1}{3}} \Big).
\]
Since $(3j-1/3)-(3j-1/2)=1/6$, by Lemma \ref{lemmacTNm} 
the non-leading terms \eqref{varpi-a} of 
$\varpi_{2j}$ contribute the following to $\widetilde R_3^+$:
\begin{equation}\label{varpi-a<<power}
O\big(M^{1-3j-2\varepsilon}T^{j+1/2+\varepsilon_1}\big)
\ \ \text{with}
\ \ \varepsilon_1=\varepsilon(3j+3/2)+3\varepsilon^2,
\end{equation}
which is 
\begin{equation}   \label{eautionvarpiminusa2j}
 \ll T^{1+\varepsilon_0}M 
~~\text{if} ~~
M \geq T^{\frac{j -1/2+ \varepsilon_1- \varepsilon_0}
{3j + 2\varepsilon}}.
\end{equation}
So we have
\[
\frac{j -1/2+ \varepsilon_1- \varepsilon_0}{3j +2\varepsilon} \leq
\frac{1}{3} - \frac{1}{6j} + 3\varepsilon
\]
for $j \geq 1$.  Thus the condition on $M$ in (\ref{eautionvarpiminusa2j}) is always true
for $M \geq T^{1/3}$.

We must now estimate the $a^{(2j)}(y_0)$ term in $\varpi_{2j}$ in
(\ref{mainterm}). Recall that $a(y)$ is given in 
\eqref{equationay}. Then 
$a^{(2j)}(y)$ will consist of a sum of terms of the following form.
Let $i_1$ be   the number of times $g(m^2 y/N)$ is differentiated (with respect to $y$)
plus the number of times a power of $y$ is differentiated.
So at every differentiation either the factor $m^2/N$ comes out,
or up to a constant, the factor $1/y$ comes out.  Notice that $1/y\asymp m^2/N$.
Let $i_2$ be the number of times $\widehat{k^*}\Big(\frac{MTc}{2 \pi^2 \sqrt{y}}\Big)$
is differentiated, and put $i_3$ to be the number of times $e\Big(\frac{-T^2c}{4 \pi^2 \sqrt{y}}\Big)$
is differentiated.  (Note that we have no restriction on the order of differentiation, and that
$a^{(2j)}(y)$ will be a sum of these terms over different possible orders of differentiation with various
coefficients.)
Then $i_1 + i_2 + i_3 = 2j$, and neglecting coefficients (which ultimately depend on $n_0$),
$a^{(2j)}(y_0)$ is bounded by the sum over all combinatorial possibilities
of
\begin{equation}\label{3powers}
\Big( \frac{N}{m^2} \Big)^{-\frac{13}{12}-i_1}  \Big( \frac{MTcm^3}{N^{\frac{3}{2}}} \Big)^{i_2}
\Big(  \frac{T^2 c m^3}{N^{\frac{3}{2}}}  \Big)^{i_3} .
\end{equation}

The main term is \eqref{3powers} when $i_3 = 2j$ and we will 
estimate this separately, below.  So we can assume in all terms 
\eqref{3powers}, now,
that
$i_1 + i_2 \geq 1$.
To estimate this error term, which is all but one term in $a^{(2j)}(y_0)$, as before, in
(\ref{mainterm}), we must divide by $\lambda_2^{j + \frac{1}{2}}$ where
 $\lambda_2 \asymp m^3N^{-3/2}/c$ with our assumption on $y_0$. 
We have then a sum of error terms which are all
 \begin{equation}\label{OMcTNm}
 O\Big( M^{i_2} c^{j + i_2 + i_3 + \frac{1}{2}} T^{i_2 + 2i_3} 
N^{\frac{3}{2}j -i_1 -\frac{3}{2}i_2 -\frac{3}{2}i_3 - \frac{1}{3}}
 m^{-3j + 2i_1 + 3i_2 + 3i_3 + \frac{2}{3}}
 \Big).
 \end{equation}
Using $i_3 = 2j - i_1 - i_2$, by Lemma \ref{lemmacTNm} this error term \eqref{OMcTNm} can be seen to be
\begin{equation}\label{equationi1i2}
\ll M^{-3j + i_1 + i_2 -\varepsilon} T^{j - i_1 - i_2 + \frac{3}{2} +\varepsilon}   
\leq T^{1 + \varepsilon_0}M 
~~\text{if}  ~~
M \geq
T^{\frac{j - i_1 - i_2 + \frac{1}{2} + \varepsilon_1 - \varepsilon_0}
{3j - i_1 - i_2 + 1 +  \varepsilon}}.
\end{equation}
Here
$\varepsilon_1 = \varepsilon (3j - i_1 + 9/2)+ 3\varepsilon^2$.
Now
\[
\frac{j-i_1-i_2+\frac{1}{2}+\varepsilon_1-\varepsilon_0}
{3j - i_1 - i_2 + 1 +  \varepsilon}
\leq \frac{j - i_1 - i_2 + \frac{1}{2}}{3j - i_1 - i_2 + 1} + 10\varepsilon
\]
We are assuming $1 \leq i_1 + i_2  \leq 2j$ with $j \geq 1$, and so
\[
\frac{j - i_1 - i_2 + \frac{1}{2}}{3j - i_1 - i_2 + 1} + 10\varepsilon \leq \frac{1}{3} -\frac{1}{6j} +10\varepsilon.
\]
Consequently, the latter condition on $M$ in (\ref{equationi1i2}) is always true for
 $M \geq T^{1/3}$.

This leaves the main term of $a^{(2j)}(y_0)$ (where $i_3 = 2j$ and $i_1 = i_2 = 0$) which is
\begin{equation}
\alpha_j \left.   \label{equationa2jalpha}
\Big( \frac{T^2c}{y^{\frac{3}{2}}} \Big)^{2j}
g\Big(\frac{m^2 y}{N}\Big)
\widehat{k^*}\Big(\frac{MTc}{2 \pi^2 \sqrt{y}}\Big)
e\Big(\frac{-T^2c}{4 \pi^2 \sqrt{y}}\Big)
y ^{-{13}/{12}} \right|_{y_0}  =:  a_{2j}(y_0).
\end{equation}
Here, the constant $\alpha_j$ depends on $j$ which ultimately can be bounded in terms of $n_0$.
If we estimate this similarly, we will get an estimate similar to
(\ref{n0estimate}) with $2j$ replacing $n_0$.  Instead, we will apply the Voronoi formula
to (\ref{equationa2jalpha}).
This is very similar to Li \cite{XLi1}, in applying the Voronoi formula a second time, but only to the main
term 
$$
\frac{e(u_2(y_0)+ {1}/{8})}{\sqrt{|u_2^{''}(y_0)|}}  a(y_0)
$$
in (\ref{mainterm}).  It appears that the term
$( T^2cy^{-\frac{3}{2}} )^{2j}$ in
(\ref{equationa2jalpha})  for $1 \leq j \leq n_0$ is on average $\asymp 1$ in summing over
$m$ and $c$, and so we do not improve upon the second application of Voronoi to the term
for just $j = 0$.

Recall that in (\ref{Psi0asymptotics}) we have 
$$
x = \frac{n_2 n_1^2}{c^3 m},
\ \ \ 
y_0 = x^2 c^6 = \frac{n_2^2 n_1^4}{m^2}. 
$$
Further, $\lambda_2 = \frac{1}{12}  c^{-1} y_0^{-\frac{3}{2}}$.
The contribution to $\widetilde{\mathcal{R}}_3^+$ of $a_{2j}(y_0)$
in (\ref{mathcalR3}) is then $\asymp \widetilde{\mathcal{R}}_{3,j}^+$
where
\begin{eqnarray}\label{equationAa2j}
 \widetilde{\mathcal{R}}_{3,j}^+  
 &=&     
 MT \sum_{m \leq C_2}  \frac{1}{m}  \sum_{c \leq C_2/m}   \frac{1}{c^{\frac{1}{2}}} \sum_{n_1 | cm}
 \sum_{n_2 >0} c \frac{A(n_1, n_2)}{n_1 n_2}
 \\
 &&\times  
 {\mathop{\sum\nolimits^*}
\limits_{u\,(\text{mod}\, mcn_1^{-1}) } }
S(0, 1+un_1 ; c)
e\Big(\frac{n_2 \bar{u}}{mcn_1^{-1}}\Big) e( -xc^2) 
x^{\frac{2}{3}} \frac{a_{2j}(y_0)}{\lambda_2^{j+ \frac{1}{2}}}.\nonumber
\end{eqnarray}
 Inserting what $x$, $y_0$, and $\lambda_2$ are in terms of $n_1$, $n_2$, $c$ and $m$ into \eqref{equationAa2j} we have
\begin{eqnarray} \label{mathcalR3jj}
\widetilde{\mathcal{R}}_{3,j}^+  
&=&    
 M T^{4j+1} \sum_{m \leq C_2} m^{3j-1} \sum_{c \leq C_2/m} c^{3j-1} \sum_{n_1 |cm} \frac{1}{n_1^{6j+1}}
 \sum_{n_2 >0} \frac{A(n_2, n_1)}{n_2^{3j+1}}
\\
&&\times
{\mathop{\sum\nolimits^*}
\limits_{u\,(\text{mod}\, mcn_1^{-1}) } }
S(0, 1+un_1 ; c)
e\Big(\frac{n_2 \bar{u}}{mcn_1^{-1}}\Big)
e\Big( -\frac{n_2 n_1^2}{cm}  \Big) \nonumber
\\
&&\times
g\Big( \frac{n_2^2 n_1^4}{N}  \Big)
\widehat{k^*} \Big( \frac{MTcm}{2\pi^2 n_2 n_1^2}  \Big)  e\Big( -\frac{T^2cm}{4\pi^2 n_2 n_1^2}  \Big).
\nonumber
\end{eqnarray}
In \eqref{mathcalR3jj} we can switch the sums over $n_2$ and 
$u$, pull out $S(0, 1+un_1 ; c)$ which does not depend on $n_2$
 and then the inner sum on $n_2$ is
 \begin{equation} \label{equationinnern2}
 \sum_{n_2 >0} A(n_2, n_1) e\Big( \frac{n_2 u'}{c'}  \Big) b_j(n_2)   
 \end{equation}
 where
\begin{equation}\label{bjy}
 b_j(y) = \frac{1}{y^{3j+1}} g\Big( \frac{y^2 n_1^4}{N}  \Big)
  \widehat{k^*} \Big( \frac{MTcm}{2\pi^2 y n_1^2}  \Big)  e\Big( -\frac{T^2cm}{4\pi^2 y n_1^2}  \Big)
\end{equation}
 and
\begin{equation}\label{uprimecprime}
\frac{u'}{c'} = \frac{\bar{u}- n_1}{mcn_1^{-1}},
~\text{with}~  (u' c') = 1  ~\text{and}~  c' | mcn_1^{-1}.
\end{equation}

We now apply the Voronoi formula for GL(3) 
(Lemma \ref{propVoronoi}) a second time 
to \eqref{equationinnern2}.  (See (4.25) of Li \cite{XLi1}.) 
We have
\begin{eqnarray}\label{equationsecondVoronoi}
&&
\sum_{n_2 \geq 1} A(n_1,n_2)    
e\Big(  \frac{n_2 u'}{c'}  \Big) b(n_2)
\\
&=&
\frac{c'}{4\pi^{{5}/{2}} i} \sum_{l_1 | c' n_1} \sum_{l_2 > 0} \frac{A(l_2, l_1)}{l_1 l_2} S(n_1\bar{u'},l_2; n_1c'l_1^{-1})
B_{0,1}^0 \Big( \frac{l_1^2 l_2}{c'^3 n_1}  \Big)
\nonumber
\\
&&+
\frac{c'}{4\pi^{{5}/{2}} i} \sum_{l_1 | c' n_1} \sum_{l_2 > 0} \frac{A(l_2, l_1)}{l_1 l_2} S(n_1\bar{u'},-l_2; n_1c'l_1^{-1})
B_{0,1}^1 \Big( \frac{l_1^2 l_2}{c'^3 n_1}  \Big).
\nonumber
\end{eqnarray}
(We followed Li \cite{XLi1} in using the notation $B$ rather than $\Psi$.)
From \eqref{equationsecondVoronoi} we have 
$x = {l_2 l_1^2}/(c'^3 n_1)$. From the function 
$g(y^2n_1^4/N)$ in \eqref{bjy} we have $X=\sqrt N/n_1^2$. 
Then 
$$
xX
=
\frac{l_2l_1^2\sqrt N}{c'^3n_1^3}
\geq
\frac{l_2l_1^2\sqrt N}{c^3m^3}
\geq
\frac{l_2l_1^2\sqrt N}{C_2^3}
\geq
l_2l_1^2T^{3/2-3\varepsilon}M^3
\gg
1
$$
by \eqref{uprimecprime}.
Consequently we can apply Lemma \ref{Vasymptotics} to $B_0(x)$ in \eqref{equationsecondVoronoi} which is, up to a negligible 
amount and lower order terms   (up to a constant)
\begin{equation}
x^{{2}/{3}} \int_0^\infty e(v_2(y)) q_j(y) ~dy   \label{B0}
\end{equation}
where
\begin{equation}\label{v2y}
v_2(y) = -3(xy)^{{1}/{3}} - \frac{T^2 c m}{4 \pi^2 y n_1^2}
\end{equation}
and
\begin{equation}\label{qjy}
q_j(y)
= y^{-3j - \frac{4}{3}}
g\Big(\frac{y^2 n_1^4}{N}\Big)
\widehat{k^*}\Big(\frac{MTcm}{2\pi^2yn_1^2}\Big).
\end{equation}
See equation (4.26) of Li \cite{XLi1}.  We need only consider the case
\[
\frac{T^6 c^3 m^3 n_1^2}{10^3 \pi^6 N^2} \leq x \leq \frac{T^6 c^3 m^3 n_1^2}{10 \pi^6 N^2}.
\]
Thus
\begin{equation}
x = \frac{l_2 l_1^2}{c'^3 n_1} \asymp \frac{T^6 c^3 m^3 n_1^2}{ \pi^6 N^2}.   \label{equationxl1l2}
\end{equation}
By the compact support of $g$, we may assume the integral (\ref{B0}) is taken over a compact
segment in $y$ so that $1 \leq y^2 n_1^4/N \leq 2$. 
With these assumptions, differentiating \eqref{v2y} we have
\[
|v''_2(y)| \gg \frac{T^2 cm n_1^4}{N^{{3}/{2}}}.
\]
By \eqref{qjy} the variation of $q_j$ over this interval can 
be seen to be $\ll y_0^{-3j-\frac{4}{3}} T^\varepsilon$.
This computation uses basic
estimates with simple calculus.  Also needed, is that
$$
y \asymp \frac{\sqrt{N}}{n_1^2},\ \
n_1 \leq cm \leq C_2 = \frac{\sqrt{N}}{T^{1-\varepsilon} M},
\ \ \text{and}\ \
M \geq T^{{1}/{3}+2\varepsilon}.
$$
Then, by the second derivative test (see Huxley \cite{Hxly}),
we have by (\ref{equationxl1l2}) that
\begin{eqnarray}\label{B0estimate}
B_0(x)
&\ll&
\Big( \frac{l_2 l_1^2}{c'^3 n_1} \Big)^{\frac{2}{3}}
\Big( \frac{T^2 c m n_1^4}{N^{-{3}/{2}}} \Big)^{-{1}/{2}}
\Big( \frac{\sqrt{N}}{n_1^2} \Big)^{-3j-\frac{4}{3}} T^\varepsilon
\\
&\ll&
T^{3+\varepsilon} c^{{3}/{2}} N^{-\frac{3}{2}j - \frac{5}{4}} n_1^{6j+2} m^{{3}/{2}}.
\nonumber
\end{eqnarray}

Put
$$
L_2 = \frac{T^6 c^3 m^3 n_1^3 c'^3}{\pi^6 N^2 l_1^2}.
$$
Combining (\ref{B0estimate}), (\ref{mathcalR3jj}), and (\ref{equationsecondVoronoi})
 we see
\begin{eqnarray}\label{R3jlongsum}
\widetilde{\mathcal{R}}_{3,j}^+
&\ll&
MT^{4j+1} \sum_{m \leq C_2} m^{3j-1}  \sum_{c \leq C_2/m}
c^{3j-1}  \sum_{n_1 | cm} \frac{1}{n_1^{6j+1}}
\sum_{u (\!\!\!\mod mcn_1^{-1})}  (1+un_1, c) c'
\\
&&\times
\sum_{l_1 | c'n_1} \sum_{l_2 \asymp L_2} \frac{|A(l_1, l_2)|}{l_1l_2}  \times \Big(  \frac{n_1c'}{l_1}  \Big)
(T^{3+\varepsilon} c^{\frac{3}{2}} N^{-\frac{3}{2}j -\frac{5}{4}} n_1^{6j+2} m^{\frac{3}{2}}).
\nonumber
\end{eqnarray}
Here $l_2 \asymp L_2$ means
${L_2}/{10^3} \leq l_2 \leq {L_2}/{10}$.  Also, we have   used the trivial bound for
the Kloosterman sum:
$$
\Big|S\Big(n_1\bar{u}, l_2;\frac{n_1 c'}{l_1}\Big)\Big|
\leq  \frac{n_1 c'}{l_1}.     
$$

Using the estimate (\ref{importantsavings}) and that
$c' \leq {mc}/{n_1}$, we deduce from \eqref{R3jlongsum} that 
\begin{equation}\label{R3jsum}
\widetilde{\mathcal{R}}_{3,j}^+  \ll N^{-\frac{3}{2}j -\frac{5}{4}} M T^{4+\varepsilon}
\sum_{m \leq C_2} m^{3j+\frac{7}{2}} \sum_{c \leq C_2/m}
c^{3j +\frac{7}{2} +\varepsilon}  \sum_{n_1 | cm} \frac{1}{n_1}
\sum_{l_1 | c'n_1} \frac{1}{l_1^2} \sum_{l_2 \asymp L_2} \frac{|A(l_1,l_2)|}{l_2}.
\end{equation} 
Now
\begin{equation}\label{ell2sum}
\sum_{l_2 \asymp L_2} \frac{|A(l_1,l_2)|}{l_2} \ll l_1 L_2^\varepsilon \ll l_1^{1-2\varepsilon}
\frac{T^{6\varepsilon} c^{6\varepsilon} m^{6\varepsilon}} 
{N^{2\varepsilon} },
\end{equation}
\begin{equation}\label{ell1sum}
 \sum_{l_1 | c'n_1} \frac{1}{l_1^{1+2\varepsilon}} 
=O(\varepsilon^{-1}),
\hspace{3mm}
\sum_{n_1 | cm} \frac{1}{n_1} 
\leq
\sum_{n_1\leq cm} \frac{1}{n_1} 
\ll
c^\varepsilon m^\varepsilon.
\end{equation}
Consequently by \eqref{ell2sum} and \eqref{ell1sum}, 
\eqref{R3jsum} is bounded by 
\[
\widetilde{\mathcal{R}}_{3,j}^+  \ll  N^{-\frac{3}{2}j -\frac{5}{4}}
M T^{4j +4+7\varepsilon} 
\sum_{m \leq C_2} m^{3j +\frac{7}{2}+7\varepsilon}
\sum_{c \leq C_2/m} c^{3j +\frac{7}{2}+8\varepsilon} .
\]
Simple calculus and similar estimates then give us
\begin{equation}\label{R3jNMTC2}
\widetilde{\mathcal{R}}_{3,j}^+  
\ll   
N^{-\frac{3}{2}j-\frac{5}{4}}
M 
T^{4j +4+7\varepsilon} 
C_2^{3j+\frac{9}{2}+8\varepsilon}.
\end{equation}
Plugging in $N= T^{3+\varepsilon}$ and 
$C_2 = \sqrt{N}/(T^{1-\varepsilon}M)$ into \eqref{R3jNMTC2}, 
we see 
\begin{equation}\label{finalestimaten0}
\widetilde{\mathcal{R}}_{3,j}^+
\ll M^{-3j -\frac{7}{2}-8\varepsilon} T^{j+\frac{5}{2} +\varepsilon_2}.   
\end{equation}
Here
$\varepsilon_2 = \varepsilon (3j + 33/2) + 12 \varepsilon^2$.
This final term (\ref{finalestimaten0}) is 
$\leq MT^{1+\varepsilon_0}$ if
\begin{equation}\label{equationR3jjfinal}
M \geq T^{\frac{j+\frac{3}{2}+ \varepsilon_2 - \varepsilon_0}
{3j + \frac{9}{2} + 8\varepsilon}}.   
\end{equation}
Now $0 \leq j \leq n_0$, and (with $0<\varepsilon\leq1/2$)
\[
\frac{j+\frac{3}{2}+ \varepsilon_2 - \varepsilon_0}
{3j + \frac{9}{2} + 8\varepsilon}
\leq 
\frac{1}{3} + \frac{3j}{3j+\frac{9}{2}}\varepsilon + 
\frac{33/2}{3j+\frac{9}{2}}\varepsilon
+\frac{12}{3j+\frac92}\varepsilon^2 
\leq \frac{1}{3} + 6\varepsilon.
\]
Thus (\ref{equationR3jjfinal}) is always true for $M \geq T^{\frac{1}{3} + 6\varepsilon}$.

Now we have showed that $\mathcal R_1^+\ll1$ after 
\eqref{R1+Bdd} and that $\mathcal R_2^+$ is negligible 
after \eqref{OTMpower}. For $\mathcal R_3^+$, other 
than negligible terms, if we take arbitrarily small 
$\varepsilon_0>0$, we have proved the bound 
$O(T^{1+\varepsilon_0}M)$ for $M\geq T^{1/3}$ 
in \eqref{eautionvarpiminusa2j} and \eqref{equationi1i2}, and 
for $M\geq T^{1/3+6\varepsilon}$ in \eqref{finalestimaten0} 
and \eqref{equationR3jjfinal}, 
where $\varepsilon>0$ is arbitrarily small independently. 
The only bound left is \eqref{MTpower} which is 
$O(T^{1+\varepsilon_0}M)$ when (\ref{n0estimate}) holds, 
where $\varepsilon>0$ is arbitrarily small as in 
\eqref{importantsavings} and 
$\varepsilon_1=\varepsilon(3n_0+4)+3\varepsilon^2$. 
To have $O(T^{1+\varepsilon_0}M)$ for any 
$M\geq T^{1/3+\varepsilon_0}$ we require 
\begin{equation}\label{<=1/3}
\frac{n_0+1+\varepsilon_1-\varepsilon_0}
{3n_0+5/2+\varepsilon}
\leq\frac13+\varepsilon_0.
\end{equation}
Solving \eqref{<=1/3} for $n_0$ we conclude that 
\eqref{MTpower} is $\ll T^{1+\varepsilon_0}M$ for 
$M\geq T^{1/3+\varepsilon_0}$ provided we take $n_0$ 
sufficiently large, i.e., if we take sufficiently many 
main terms in \eqref{mainterm} when we apply Proposition 
\ref{WSPI}:
\begin{equation}\label{n0large}
n_0\geq
\frac1{\varepsilon_0-\varepsilon}
\Big(
\frac1{18}+\frac{11\varepsilon}9
-\frac{7\varepsilon_0}6+\varepsilon^2
-\frac{\varepsilon\varepsilon_0}3
\Big).
\end{equation}
Here we may simply take $\varepsilon=\varepsilon_0/6$. 

Therefore, we have proved that $\mathcal R^+$ in 
\eqref{equationRplus} is bounded by $T^{1+\varepsilon_0}M$ 
for $M\geq T^{1/3+\varepsilon_0}$ by choosing $n_0$ 
satisfying \eqref{n0large} and setting the $\varepsilon$ 
in \eqref{finalestimaten0} equal to $\varepsilon_0/6$.

\section{$K$-Bessel function terms}
Following Li \cite{XLi1} we split $\mathcal{R}^-$ as in 
\eqref{equationRminus} into $\mathcal{R}_1^- + \mathcal{R}_2^- $ with
\begin{eqnarray}\label{R1minus}
\mathcal{R}_1^- 
&=&
\sum_{m \geq 1} \sum_{n \geq 1} \frac{A(m,n)}{(m^2n)^{\frac{1}{2}}}
g\Big(\frac{m^2n}{N}\Big)
\sum_{c \geq C/m} c^{-1}S(n, -1; c) H_{m,n}^- \Big( \frac{4\pi\sqrt{n}}{c}  \Big)  ,  
\\
\label{R2minus}
\mathcal{R}_2^- 
&=&
\sum_{m \geq 1} \sum_{n \geq 1} \frac{A(m,n)}{(m^2n)^{\frac{1}{2}}}
g\Big(\frac{m^2n}{N}\Big)
\sum_{c \leq C/m} c^{-1}S(n, -1; c) H_{m,n}^- \Big( \frac{4\pi\sqrt{n}}{c}  \Big)   ,   
\end{eqnarray}
where $H_{m,n}^-$ is defined in \eqref{Hmnminus} and 
$C = \sqrt{N} + T$. 
In estimating $\mathcal{R}_1^-$, one can express the $K$-Bessel function in terms of the $I$-Bessel
function.  Set $\sigma = 100$.
Then the estimates for the $I$-Bessel function, along with Li's previous estimates of $V$
(see (4.7) and (5.6) of Li \cite{XLi1}) give a bound for 
\eqref{R1minus} (using the trivial bound for the Kloosterman sum)
\begin{equation}\label{R1minus<<}
\mathcal{R}_1^- \ll MT^{\sigma+1+\varepsilon}
\sum_{m \leq \sqrt{2N}} \frac{1}{m^{1+2\sigma}}  \sum_{n \leq \frac{2N}{m^2}} \frac{A(m,n)}{n^{\frac{1}{2}}}
\sum_{c \geq C/m} \frac{1}{c^{2\sigma}}  e^{4\pi \frac{\sqrt{n}}{c}}.
\end{equation}
Using $n\leq2N/m^2$ and $c \geq C/m$ we see that 
$e^{4\pi\sqrt n/c} \ll 1$.
Further,
\[
\sum_{c \geq C/m} \frac{1}{c^{2\sigma}}
\ll
\Big(\frac Cm\Big)^{1-2\sigma}~~~~\text{and}~~~~
 \sum_{n \leq \frac{2N}{m^2}} \frac{A(m,n)}{n^{\frac{1}{2}}} \ll m\Big( \frac{2N}{m^2}  \Big)^{\frac{1}{2}}.
\]
Plugging this into \eqref{R1minus<<}, and noting the sum over 
$m$ converges, we have
\begin{equation}\label{R1minus<<1}
\mathcal{R}_1^- \ll  \sqrt{N} M T^{\sigma + 1 + \varepsilon  } C^{1-2\sigma}   \ll  1
\end{equation}
for $\varepsilon$ sufficiently small.  Notice this bound holds for $T^\varepsilon \leq M \leq T^{1-\varepsilon}$.

Following the derivation in Li \cite{XLi1}, up to a negligible term, we can write
\begin{equation}\label{Hmnminussum}
H_{m,n}^-(x) = H_{m,n}^{-,1}(x) + H_{m,n}^{-,2}(x)
\end{equation}
where
\begin{eqnarray*}
H_{m,n}^{-,j}(x)
&=&
\frac{4M^j T^{2-j}}{\pi}
\int_\mathbb{R} \int_{|\zeta|\leq T^\varepsilon}
t^{j-1} e^{-t^2} V(m^2 n, tM+T)
\\
&&
\times\cos(x\sinh \zeta) e\Big( -\frac{(tM+T)\zeta}{\pi}  \Big) ~dt d\zeta,
\end{eqnarray*}
for $j = 1, 2$. In \eqref{Hmnminussum} 
$H_{m,n}^{-,2}(x)$ is a lower order term.  We only work with $H_{m,n}^{-,1}(x)$, since the
analysis with $H_{m,n}^{-,2}(x)$ is similar.
Up to a negligible amount, we can write
$H_{m,n}^{-,1}(x) = 4Y_{m,n}(x)$, 
where
\[
Y_{m,n}(x) = \frac{Y_{m,n}^*(x) + Y_{m,n}^*(-x)}{2},
\]
with
\begin{equation}
Y_{m,n}^*(x) = T\int_\mathbb{R} \widehat{k^*}(\zeta) e\Big( -\frac{T\zeta}{M} + \frac{x}{2\pi} \sinh \frac{\zeta \pi}{M}
\Big)   ~d\zeta.   \label{Ymn}
\end{equation}
The part of the integral over $|\zeta| \geq M^{\varepsilon/2}$ in \eqref{Ymn} is negligible.
Further, with this assumption, it can be shown by integration by parts, that
$Y_{m,n}^*(x)$ is negligible unless
\begin{equation}\label{eqx100T}
\frac{T}{100} \leq |x| \leq 100T~~~~\text{and}~~~~\frac{x}{M^3} \ll T^{-\varepsilon},   
\end{equation}
which we now assume. Recall 
$M \geq T^{\frac{1}{3}+2\varepsilon}$.  Thus, the sum over 
$c$ in (\ref{R2minus}) for which 
$$
c \geq \frac{400\pi \sqrt{N}}{Tm}
\ \ \text{or}\ \ c \leq \frac{\sqrt{2}\pi \sqrt{N}}{25Tm}
$$
is negligible.  We thus may assume 
$$
\frac{\sqrt{2}\pi \sqrt{N}}{25Tm} \leq c \leq \frac{400\pi \sqrt{N}}{Tm}
$$
and we will denote this by $c \asymp\sqrt N/(Tm)$.

Using one more nonzero term in the Taylor expansion than
Li \cite{XLi1}, estimating, we have
\begin{eqnarray}\label{Ymn5terms}
Y_{m,n}^*(x)
&=&
T\int_\mathbb{R} \widehat{k^*}(\zeta)
e\Big( -\frac{T\zeta}{M} + \frac{x\zeta}{2M} + \frac{\pi^2 x \zeta^3}{12M^3}
+\frac{\pi^4 x \zeta^5}{240M^5} +\frac{\pi^6 x \zeta^7}{2\cdot 7! M^7}         \Big)   ~d\zeta
\\
&&+
O\Big(  T\int_\mathbb{R} |\widehat{k*} (\zeta)| \frac{|\zeta|^9 |x|}{M^9}   \Big).
\nonumber
\end{eqnarray}
Now, expanding
$$
e\Big(  \frac{\pi^2 x \zeta^3}{12M^3}
+\frac{\pi^4 x \zeta^5}{240M^5} +\frac{\pi^6 x \zeta^7}{2\cdot 7! M^7}         \Big)
$$
in \eqref{Ymn5terms} into a Taylor series of order
$L_2$ (which could depend on $\varepsilon$) we have
\begin{eqnarray*}
Y_{m,n}^*(x)
&=&
T\int_\mathbb{R} \widehat{k^*}(\zeta) e\Big(  -\frac{(2T-x)\zeta}{2M}    \Big)
\\
&&
\times
\sum_{j_1 + j_2 + j_3 \leq L_2} d_{j_1 , j_2 , j_3} \Big( \frac{x\zeta^3}{M^3} \Big)^{j_1}
\Big( \frac{x\zeta^5}{M^5} \Big)^{j_2} \Big( \frac{x\zeta^7}{M^7} \Big)^{j_3}~d\zeta
+ O\Big( \frac{T|x|^{L_2+1}}{M^{3L_2+3}} + \frac{T|x|}{M^9}   \Big),
\end{eqnarray*}
where $d_{j_1 , j_2 , j_3}$ are constants with $d_{0,0,0} = 1$ with the sum taken over
$j_1 \geq 0$, $j_2 \geq 0$, and $j_3 \geq 0$.
It follows that
\begin{eqnarray}\label{eqkderivative}
Y_{m,n}^*(x) 
&=& T \sum_{j_1 + j_2 + j_3 \leq L_2} \frac{d_{j_1 , j_2 , j_3} \cdot   
x^{j_1 + j_2 + j_3}}{(2\pi i M)^{3j_1 + 5j_2 + 7j_3 }}
k^{*(3j_1 + 5j_2 + 7j_3)} \Big( \frac{x-2T}{2M}  \Big)
\\
&&+ 
O \Big(   \frac{T|x|^{L_2+1}}{M^{3L_2+3}}\Big) + 
O\Big(\frac{T|x|}{M^9}  \Big).\nonumber
\end{eqnarray}

We take $L_2$ large enough (possibly depending on $\varepsilon$) so that the first error term in \eqref{eqkderivative} is 
negligible, or rather has as fast inverse polynomial decay as 
desired. (Recall \eqref{eqx100T}.) The contribution to 
$\mathcal{R}_2^-$ coming from the error term
$O \big( {T|x|}/{M^9}     \big)$
can be seen to be bounded by
\begin{equation}\label{T2M9}
\frac{T^2}{M^9} \sum_{m \leq \sqrt{2N}} \frac{1}{m}
\sum_{n \leq {2N}/{m^2}} \frac{|A(m,n)|}{n^{\frac{1}{2}}}
\sum_{c \leq C/m}  \frac{|S(n,-1;c)|}{c}.
\end{equation}
Using Weil's bound for $S(n,-1;c)$ we see
\[
\sum_{c \leq C/m}  \frac{|S(n,-1;c)|}{c}
\ll \Big(\frac Cm\Big)^{\frac{1}{2}+\varepsilon}.
\]
Estimating similarly to the above, we see that 
\eqref{T2M9} is bounded by 
\[
\ll \frac{T^2}{M^9} C^{\frac{1}{2}+\varepsilon} \sqrt{N} = \frac{T^{2+ \frac{3}{2} +\frac{3}{4}+ \varepsilon}}{M^9} .
\]
The above is $\ll T^{1+\varepsilon}M$ by a power of $T$ for $M \geq T^{\frac{1}{3}+2\varepsilon}$.

We take the leading term in the finite series for $Y_{m,n}^*(x)$ in (\ref{eqkderivative}),
as the terms with higher derivatives of $k^*$  can be handled in the same way.  It follows we need to bound
\begin{equation}\label{simR2minus}
\widetilde{\mathcal{R}}_2^- =
T\sum_{m\geq 1} \sum_{n \geq 1}
\frac{A(m,n)}{(m^2 n)^{\frac{1}{2}}}
g\Big(\frac{m^2 n}{N}\Big)
\sum_{c \asymp \frac{\sqrt{N}}{Tm}}
\frac{S(n,-1;c)}{c}k^*\Big(  \frac{{4\pi \sqrt{n}}/{c}-2T}{2M}  \Big).    
\end{equation}
Denote 
\[
r(y) = g\Big(\frac{m^2 y}{N}\Big)
k^*\Big(  \frac{{4\pi \sqrt{y}}/{c}-2T}{2M}  \Big) y^{-\frac{1}{2}},
\]
which is a smooth function of compact support. From 
$x = {n_2 n_1^2}/(c^3 m)$ and $X=N/m^2$ we know 
$$
xX= \frac{n_2 n_1^2N}{c^3m^3} 
\geq \frac{n_2 n_1^2N}{C^3} 
\geq T^{\frac{3}{2} -\varepsilon}
\gg1.
$$
Consequently we may apply the Voronoi formula (Lemma 
\ref{propVoronoi}) and its asymptotic expansion (Lemma 
\ref{Vasymptotics}) to the sum over $n$ in \eqref{simR2minus}. 
As in Li \cite{XLi1} we only consider $R_0(x)$ (see (5.11) of \cite{XLi1}), which is (up to lower order terms)
\[
R_0(x) = 2\pi^4 x i \int_0^\infty r(y) \frac{d_1\sin(6\pi (xy)^{\frac{1}{3}})}{\pi (xy)^{\frac{1}{3}}} ~dy.
\]
Li \cite{XLi1} states that (in an equivalent form) if
$n_2 \gg \frac{N^{\frac{1}{2}} T^\varepsilon}{M^3 n_1^2}$,
then
$r'(y)x^{-\frac{1}{3}} y^{\frac{2}{3}} \ll T^{-\varepsilon}$. 
For this assumption on $n_2$, the integral term in $R_0$ as well as the contribution to
$\widetilde{\mathcal{R}}_2^-$  is found to be negligible.

Thus, we may assume
$n_2 \ll \frac{N^{\frac{1}{2}} T^\varepsilon}{M^3 n_1^2}$.
Now, $r(y)$ is negligible unless
$$
\Big|  \frac{{2\pi\sqrt{y}}/{c}-T}{M}   \Big| 
\leq T^\varepsilon.
$$
This gives us an interval of width $\ll T^{1+\varepsilon} M c^2$ where $y \asymp {N}/{m^2}$, and so
\[
R_0(x)
\ll
\Big(\frac{n_2 n_1^2}{c^3 m}\Big)^\frac{2}{3}
\Big(\frac N{m^2}\Big)^{-\frac{5}{6}} T^{1+\varepsilon} M c^2.
\]
Using this estimate along with (\ref{importantsavings}) it follows from \eqref{simR2minus} that
\begin{eqnarray}\label{R2-<<}
\widetilde{\mathcal{R}}_2^-
&\ll&
T
\sum_{m\leq \sqrt{2N}}
\sum_{c \asymp \frac{\sqrt{N}}{Tm}}
\sum_{n_1 | cm}
\sum_{n_2 \ll {\sqrt{N}T^\varepsilon}/(M^3 n_1^2)}
\frac{|A(n_1, n_2)|}{n_1 n_2}
\frac{mc^{1+\varepsilon}}{n_1}
\\
&&\times
\Big(\frac{n_2 n_1^2}{c^3 m}\Big)^\frac{2}{3}
\Big(\frac N{m^2}\Big)^{-\frac{5}{6}}
T^{1+\varepsilon} M c^2
\nonumber
\\
&=&
T^{2+\varepsilon} M N^{-\frac{5}{6}} \sum_{m\leq \sqrt{2N}} m \sum_{c \asymp C/m} c^{1+\varepsilon}
\sum_{n_1 | cm} n_1^{-\frac{2}{3}}
\sum_{n_2 \ll {\sqrt{N}T^\varepsilon}/(M^3 n_1^2) }
\frac{|A(n_1, n_2)|}{n_2^{\frac{1}{3}}}.
\nonumber
\end{eqnarray}
Estimating similarly to the last section, the inner sum 
in \eqref{R2-<<} is 
\[
 \sum_{n_2 \ll {\sqrt{N}T^\varepsilon}/(M^3 n_1^2) }
\frac{|A(n_1, n_2)|}{n_2^{{1}/{3}}}
\ll
 n_1
\Big( \frac{\sqrt{N}T^\varepsilon}{M^3 n_1^2} \Big)^{{2}/{3}}.
\]
Plugging this and 
\[
\sum_{n_1 | cm} \frac{1}{n_1} \ll (cm)^\varepsilon  .
\]
into \eqref{R2-<<} we see
\[
\widetilde{\mathcal{R}}_2^-
\ll
T^{2+{5\varepsilon}/{3}} M^{-1} N^{-{1}/{2}}
\sum_{m \leq \sqrt{2N}} m^{1+\varepsilon}
\sum_{c \asymp \frac{\sqrt{N}}{Tm}} c^{1+2\varepsilon}  .
\]
Now
\[
\sum_{c \asymp \frac{\sqrt{N}}{Tm}} c^{1+2\varepsilon}
\ll
\Big(\frac{\sqrt{N}}{Tm} \Big)^{2+2\varepsilon}  ~~~~\text{and}~~~~
\sum_{m \leq \sqrt{2N}} \frac{1}{m^{1+\varepsilon}}  \ll  \frac{1}{\varepsilon}.
\]
Consequently,
$\widetilde{\mathcal{R}}_2^- \ll  T^{\frac{3}{2}  +\frac{13}{6} \varepsilon} M^{-1}$. 
This is clearly smaller than $T^{1+\varepsilon_0}M$ 
if $M\geq T^{1/4+13\varepsilon/12-\varepsilon_0/2}$. 

Together with \eqref{R1minus<<1} for 
$T^\varepsilon \leq M \leq T^{1-\varepsilon}$, we conclude 
that $\mathcal R^-\ll T^{1+\varepsilon_0}M$ if 
$M\geq T^{1/3}$. Recall that $\mathcal D$ in \eqref{Ddefine} 
is negligible for $T^\varepsilon \leq M \leq T^{1-\varepsilon}$ 
as we pointed at the end of Section 3. Together with our 
conclusion at the end of Section 4 for $\mathcal R^+$, we 
have proved that $\mathcal R$ in \eqref{Rdefine} is bounded 
by $O(T^{1+\varepsilon_0}M)$ for 
$T^{1/3+\varepsilon_0}\leq M\leq T^{1/2}$. 
This implies Theorem \ref{SumIntBdd}. 
\qed

\end{document}